\documentclass{aims}
\usepackage{amsmath}
\usepackage{paralist}
\usepackage{graphics} 
\usepackage{epsfig} 
\usepackage{graphicx}  \usepackage{epstopdf} 
\usepackage[colorlinks=true]{hyperref}
\hypersetup{urlcolor=blue, citecolor=red}

\def\currentvolume{X}
\def\currentissue{X}
\def\currentyear{200X}
\def\currentmonth{XX}
\def\ppages{X--XX}

\newtheorem{theorem}{Theorem}[section]
\newtheorem{corollary}{Corollary}

\newtheorem{lemma}[theorem]{Lemma}
\newtheorem{proposition}{Proposition}

\theoremstyle{definition}
\newtheorem{definition}[theorem]{Definition}
\newtheorem{remark}{Remark}

\usepackage[utf8]{inputenc}
\usepackage{amssymb}
\usepackage{amsmath,bbm,bm}
\usepackage{amsfonts}
\usepackage{mathrsfs}
\usepackage{graphicx}
\usepackage{enumerate}
\usepackage{color}
\usepackage{cite}
\usepackage{soul}

\newtheorem{assumption}[theorem]{Assumption}

\DeclareMathAlphabet{\mathpzc}{OT1}{pzc}{m}{it}
\newcommand{\cEl}{\scalebox{1.25}{$\mathpzc{E}$}}
\newcommand{\cBl}{\scalebox{1.25}{$\mathpzc{B}$}}
\newcommand{\cAl}{\scalebox{1.25}{$\mathpzc{A}$}}
\newcommand{\cCl}{\scalebox{1.25}{$\mathpzc{C}$}}
\newcommand{\R}{\mathbb{R}}

\DeclareMathOperator{\curl}{curl}
\DeclareMathOperator{\divg}{div}
\DeclareMathOperator{\im}{im}
\newcommand{\scpr}[2]{\left\langle #1,#2\right\rangle}

\title[Passivity and solution estimates for a coupled MQS model]
{Passivity, Port-Hamiltonian Formulation and Solution Estimates
		for a Coupled Magneto-Quasistatic System}

\author[Timo Reis and Tatjana Stykel]{}


\email{timo.reis@tu-ilmenau.de}
\email{stykel@math.uni-augsburg.de}

\begin{document}
\maketitle

\centerline{\scshape Timo Reis$^*$}
\medskip
{\footnotesize
	\centerline{Institut für Mathematik} 
	\centerline{Technische Universit\"at Ilmenau}
	\centerline{Weimarer Str. 32}
	\centerline{98693 Ilmenau, Germany}
} 

\medskip

\centerline{\scshape Tatjana Stykel}
\medskip
{\footnotesize
	\centerline{Institut f\"ur Mathematik \& Centre for Advanced Analytics and Predictive Sciences}
	\centerline{Universit\"at Augsburg}
	\centerline{Universit\"atsstr. 12a}
	\centerline{86159 Augsburg, Germany}
}

\bigskip
\centerline{(Communicated by the associate editor name)}

\begin{abstract}
{In this paper,   we study a~quasilinear coupled magneto-quasistatic model from a~systems theoretic perspective.}
First, by taking the injected voltages as input and the associated currents as output, we prove that the magneto-quasistatic system is passive. Moreover, by defining suitable Dirac and resistive structures, we show that it admits a~representation as a~port-Hamiltonian system. Thereafter, we consider dependence on initial and input data. We show that the current and the magnetic vector potential can be estimated by means of the initial magnetic vector potential and the voltage. We also analyse the free dynamics of the system and study the asymptotic behavior of the solutions for $t\to\infty$.
\end{abstract}

\section{Introduction}\label{sec:intro}

A~subject of this paper is the investigation of systems theoretic properties of a~quasulinear coupled magneto-quasistatic (MQS) model
\begin{subequations}\label{eq:MQS}
	\begin{align}
	{\tfrac{\partial}{\partial t}}\left(\sigma\bm{A}\right) + \nabla \times \left(\nu(\cdot,\|\nabla \times \bm{A}\|_2)
	\nabla \times \bm{A}\right)  = & \; \chi\,\bm{i} & \text{ in } &\mathit{\Omega}\times (0,T), \label{eq:MQS1} \\
	\tfrac{{\rm d}}{{\rm d} t} \int_\mathit{\Omega} \chi^\top\bm{A} \, {\rm d}\xi + R\, \bm{i} = &\;\bm{v}\label{eq:MQScoupl} & \text{ on } &(0,T), \\
	\bm{A}\times \bm{n}_o  = &\; 0 & \mbox{in }& \partial \mathit{\Omega}\times (0,T),
	\label{eq:MQSbc}\\[2mm]
	\sigma\bm{A}(\cdot,0) = &\; \sigma\bm{A}_0 &\text{ in }&\mathit{\Omega},
	\label{eq:MQSic1}\\
	\int_\mathit{\Omega} \chi^\top\bm{A}(\cdot,0) \, {\rm d}\xi=&\, \int_\mathit{\Omega} \chi^\top\bm{A}_0\,
	{\rm d}\xi,&& \label{eq:MQSic2}
	\end{align}
\end{subequations}
where $\mathit{\Omega}\subset\R^3$ is a~bounded domain with some further properties specified later on, $\bm{A}:\overline{\mathit{\Omega}}\times [0,T]\to\mathbb{R}^3$ is the magnetic vector potential,
$\nu:\mathit{\Omega}\times\mathbb{R}_{\ge0}\to\mathbb{R}_{\ge0}$ is the magnetic reluctivity,
$\sigma:\mathit{\Omega}\to\mathbb{R}_{\ge0}$ is the electric conductivity, and
\mbox{$\bm{v},\bm{i}:{[0,T]}\to\mathbb{R}^m$} are, respectively, the voltage and the electrical current through the electromagnetic conductive contacts. Furthermore,
$\chi:\mathit{\Omega}\to\mathbb{R}^{3\times m}$ is the winding function, which expresses the geometry of $m$ windings,
$R\in\mathbb{R}^{m\times m}$ is the resistance of the winding,
{$\bm{n}_o:\partial \mathit{\Omega}\to\mathbb{R}^3$ is the outward unit normal vector to the boun\-da\-ry $\partial \mathit{\Omega}$,}
and \mbox{$\bm{A}_0: \mathit{\Omega}\to\mathbb{R}^3$}
is the initial value for the magnetic vector potential. The voltage $\bm{v}$ is considered as the input, whereas the current $\bm{i}$ is the output of the system.
The dynamics of the magnetic vector potential is expressed by \eqref{eq:MQS1}, which arises from a~simplification of Maxwell's equations \cite{ChReSt21}, and \eqref{eq:MQSbc}
means that the magnetic flux through the boundary $\partial \mathit{\Omega}$ is zero.
The term $\chi\,\bm{i}$ stands for the external current induced by $m$~windings, and \eqref{eq:MQScoupl} expresses Kirchhoff's voltage law including the applied voltage $\bm{v}$, the resistive voltage~$R\,\bm{i}$, and the term $\tfrac{{\rm d}}{{\rm d} t} \int_\mathit{\Omega} \chi^\top\bm{A} \, {\rm d}\xi$, which is the voltage induced by the electromagnetic field.
Note that $\sigma$ is positive and constant on a~subdomain $\mathit{\Omega}_C\subset\mathit{\Omega}$, whereas it vanishes on the complement of~$\mathit{\Omega}_C$. This implies that the coupled MQS system \eqref{eq:MQS} becomes of degenerate parabolic type.

{Existence, uniqueness and regularity properties of solutions to \eqref{eq:MQS} have been studied in \cite{ArnH12,NicT14,PauPTW21} for some special cases, while \cite{ChReSt21}  presents well-posedness and regularity results for more general setting.}
The purpose of this paper is to investigate the dynamic behavior of the coupled MQS model \eqref{eq:MQS} from a~systems theoretic point of view. In doing so, we analyse passivity of this model by establishing the existence of a~certain energy balance. Further, we  show that the MQS system \eqref{eq:MQS} fits
into the framework of port-Hamiltonian systems. The findings on passivity will be the basis for solution estimates by means of the initial value $\bm{A}_0$ and the input voltage~$\bm{v}$. Besides showing that the free dynamics of~\eqref{eq:MQS} (that is, the solution behavior with $\bm{v}\equiv 0$) are bounded and that the $L^2$-norm of the curl of the magnetic vector potential decays exponentially, we present estimates for
the $L^2$-norm of the output and the magnetic vector potential in terms of the initial value, the input and the material parameters.

The paper is organized as follows.
In Section~\ref{ssec:mqs}, we collect our model assumptions and review {some results from \cite{ChReSt21}} on existence, uniqueness and regularity of solutions of the  coupled MQS system \eqref{eq:MQS}. We also highlight that this system admits a~representation as a~differenial-algebraic system in which the state evolves in an~infinite-dimensional Hilbert space.
This section closes with a~brief introduction to the concept of magnetic energy. It plays an~essential role in Section~\ref{sec:passivity}, where we establish an~energy balance of the MQS system \eqref{eq:MQS} and prove that this system is passive. In Section~\ref{sec:pH}, we introduce infinite-dimensional port-Hamiltonian systems via abstract Dirac and resistive structures and show that the coupled MQS system~\eqref{eq:MQS} belongs to this class. The energy balance from Section~\ref{sec:passivity} is then used in Section~\ref{sec:stability}
to establish the solution estimates for \eqref{eq:MQS}. In Section~\ref{sec:io}, we derive a~bound for the $L^2$-norm of the output in terms of the $L^2$-norm of the input and the initial value. This gives, in particular, an~estimate for the generalization of the \mbox{$\mathcal{H}_\infty$-norm} to the nonlinear infinite-dimensional differential-algebraic case. In Section~\ref{sec:is}, we derive estimates for the magnetic energy and the $L^2$-norm of the magnetic vector potential at some given time $t\geq0$ by means of the initial value and the $L^2$-norm of the input. Finally, in Section~\ref{sec:uncont}, we consider the MQS system~\eqref{eq:MQS} in which the zero voltage is applied. We show that the magnetic energy decays exponentially, and we present estimates for the $L^2$-norm of the magnetic
vector potential at time $t\geq0$.

\section{The MQS System: Assumptions, Solvability and Magnetic Energy}
\label{ssec:mqs}

Throughout this paper, we use the notation and terminology {from} \cite{ChReSt21}.
For ease of reference, we recall here the assumptions on the spatial domain,
the material parameters, the initial value, and the winding function. Moreover, we present
the solution concept and recap the solvability results and some essential properties
of the magnetic energy established in \cite{ChReSt21} {which will be required in the following}.

\subsection{The model assumptions}
First, we impose {some} assumptions on the spatial domain $\mathit{\Omega}$.

\begin{assumption}[Spatial domain, geometry and topology]
	\label{ass:omega}
{\em	The set $\mathit{\Omega}\subset\mathbb{R}^3$ is a~simply connected bounded Lipschitz domain, which is decomposed into two Lipschitz regular, open subsets \mbox{$\mathit{\Omega}_{C}$, $\mathit{\Omega}_{I}{\subset} \mathit{\Omega}$}, called, respectively, {\em conducting} and {\em non-conducting subdomains}, such that
	$\overline{\mathit{\Omega}}_{C}{\subset} {\mathit{\Omega}}$ and
	$\mathit{\Omega}_{I}=\mathit{\Omega}\setminus \overline{\mathit{\Omega}}_C$. Furthermore, $\mathit{\Omega}_C$ is connected, and $\mathit{\Omega}_{I}$ has finitely many connected {\em internal subdomains} $\mathit{\Omega}_{I,1},\ldots,\mathit{\Omega}_{I,q}$ with single boundary components $\Gamma_1,\ldots,\Gamma_q$, respectively,
	and the {\em external subdomain} $\mathit{\Omega}_{I,{\rm ext}}$ which has two boundary components
	$\Gamma_{{\rm ext}}=\overline{\mathit{\Omega}}_{I,{\rm ext}}\cap \overline{\mathit{\Omega}}_{C}$ and $\partial\mathit{\Omega}$.}
\end{assumption}

Next, we state the assumptions on the electric conductivity $\sigma$,
the magnetic reluctivity~$\nu$, and the resistance matrix $R$.

\begin{assumption}[Material parameters]\!\!\label{ass:material}
	{\em
		\begin{enumerate}[\rm a)]
		\item \label{ass:material1}
		The electric conductivity $\sigma:\mathit{\Omega}\to\mathbb{R}_{\geq 0}$ is of the form
		$\sigma=\sigma_C\mathbbm{1}_{\mathit{\Omega}_C}$, where $\sigma_C>0$ and
		$\mathbbm{1}_{\mathit{\Omega}_C}$ denotes the indicator function of the subdomain $\mathit{\Omega}_C$.
		\item \label{ass:material2}
		The magnetic reluctivity $\nu:\mathit{\Omega}\times \mathbb{R}_{\ge0}\to\mathbb{R}_{\ge0}$ has the following properties:
		\begin{enumerate}[\rm (i)]
			\item\label{ass:material2a} $\nu$ is measurable;
			\item\label{ass:material2c} the function $\zeta\mapsto\nu(\xi,\zeta)\zeta$ is strongly monotone with a~monotonicity constant \mbox{$m_{\nu}>0$} independent of $\xi\in\mathit{\Omega}$. In other words, there exists $m_{\nu}>0$ such that
			\[
			\bigl(\nu(\xi,\zeta) \zeta-\nu(\xi,\varsigma)\varsigma\bigr)(\zeta-\varsigma)\geq m_{\nu} (\zeta-\varsigma)^2 \enskip\text{ for all } \xi\in\mathit{\Omega}, \enskip \zeta,\varsigma\in\mathbb{R}_{\ge0};
			\]
			\item\label{ass:material2d}
			the function $\zeta\mapsto\nu(\xi,\zeta)\zeta$ is Lipschitz continuous with a~Lipschitz constant $L_{\nu}>0$ independent of $\xi\in\mathit{\Omega}$. In other words, there exists $L_{\nu}>0$ such that
			\[
			|\nu(\xi,\zeta)\zeta-\nu(\xi,\varsigma)\varsigma| \leq L_{\nu} |\zeta-\varsigma|\quad\text{ for all } \xi\in\mathit{\Omega}, \enskip\zeta,\varsigma\in\mathbb{R}_{\ge0}.
			\]
		\end{enumerate}
		\item\label{ass:resistance} The resistance matrix $R\in\mathbb{R}^{m\times m}$ is symmetric and positive definite.
	\end{enumerate}
	}
\end{assumption}

As the space in which the solutions of the coupled MQS system \eqref{eq:MQS} evolve, we consider the set $X(\mathit{\Omega},\mathit{\Omega}_C)$ of all square integrable functions which are $L^2$-orthogonal to all gradient fields of functions from $H^1_0(\mathit{\Omega})$ being constant on each interface component $\Gamma_1,\ldots,\Gamma_q$ and $\Gamma_{\rm ext}$.
The space $X(\mathit{\Omega},\mathit{\Omega}_C)$ is a~Hilbert space equipped with the standard inner product in ${L^2(\mathit{\Omega};\mathbb{R}^3)}$.
We further consider the space
\begin{equation}
X_0(\curl,\mathit{\Omega},\mathit{\Omega}_C)=H_0(\curl,\mathit{\Omega})\cap X(\mathit{\Omega},\mathit{\Omega}_C),
\label{eq:statespace2}
\end{equation}
which is again a Hilbert space, now provided with the inner product in $H_0(\curl,\mathit{\Omega})$.
For any $\bm{A}\in X(\mathit{\Omega},\mathit{\Omega}_C)$, one has $\mathbbm{1}_{\mathit{\Omega}_C} \bm{A}\in X(\mathit{\Omega},\mathit{\Omega}_C)$, see \cite[Lemma~3.3]{ChReSt21}.
Moreover, \cite[Lemma~3.4]{ChReSt21} states that $X_0(\curl,\mathit{\Omega},\mathit{\Omega}_C)$ is dense in $X(\mathit{\Omega},\mathit{\Omega}_C)$.

The space $X(\mathit{\Omega},\mathit{\Omega}_C)$ enables us to formulate the assumptions on the initial magnetic vector potential $\bm{A}_0$ and the winding function $\chi$. For the latter, we impose a~condition for which it is necessary that all components of the matrix-valued function $\chi:\Omega\to\R^{3\times m}$ are square integrable, i.e., $\chi\in L^2(\Omega;\R^{3\times m})$. Note that
such a~function can, loosely speaking, be canonically identified with a~$m$-tuple of elements of $L^2(\Omega;\R^{3})$. More precisely, we will use the identification
\begin{equation}\label{eq:L23m}
L^2(\Omega;\R^{3\times m})\cong L^2(\Omega;\R^{3})^{1\times m},
\end{equation}
where elements of $L^2(\Omega;\R^{3})^{1\times m}$ are regarded as operators from $\R^m$ to $L^2(\Omega;\R^{3})$. The corresponding operator norm is given by
\begin{equation}\label{eq:chinorm}
\|\chi\|_{L^2(\mathit{\Omega};\mathbb{R}^{3\times m})}=\sqrt{\lambda_{\max}\left(\int_{\mathit{\Omega}} \chi^\top \chi \, {\rm d}\xi\right)},\end{equation}
where $\lambda_{\max}(\mathcal{X})$ stands for the largest eigenvalue of~a~symmetric matrix $\mathcal{X}\in\R^{m\times m}$. Note that the resulting normed space $L^2(\mathit{\Omega};\mathbb{R}^{3\times m})$ is not a~Hilbert space unless $m=1$, since the above norm is not induced by an~inner product.

\begin{assumption}[Initial condition and winding function]\label{ass:init}\
{\em	\begin{enumerate}[\rm a)]
		\item\label{ass:initial}
		The initial magnetic vector potential $\bm{A}_0:\mathit{\Omega}\to\mathbb{R}^3$ belongs to
		$X(\mathit{\Omega},\mathit{\Omega}_C)$.
		\item\label{ass:winding}
		The winding function $\chi:\mathit{\Omega}\to\mathbb{R}^{3\times m}$ belongs, by using the identification \eqref{eq:L23m}, to $X(\mathit{\Omega},\mathit{\Omega}_C)^{1\times m}$.
	\end{enumerate}
	}
\end{assumption}

\subsection{Solutions of the coupled MQS system}
\label{sec:solution}

Before discussing the solution properties of the coupled MQS system \eqref{eq:MQS}, we declare what we mean by solutions.
Let $T>0$ and $\bm{v}\in L^2([0,T];\mathbb{R}^m)$.
We call  $(\bm{A},\bm{i})$ with $\bm{A}\in\overline{\mathit{\Omega}}\times [0,T]\to\mathbb{R}^3$ and $\bm{i}:[0,T]\to\mathbb{R}^m$ a~{\em weak solution} (or just {\em solution})
of the coupled MQS system~\eqref{eq:MQS}, if
\begin{enumerate}[\rm a)]
	\item\label{item:sol1} $\sigma\bm{A}\in C([0,T];X(\mathit{\Omega},\mathit{\Omega}_C))\cap
	H_{\rm loc}^1((0,T];X(\mathit{\Omega},\mathit{\Omega}_C))$ and $\sigma\bm{A}(0)=\sigma\bm{A}_0$,
	\vspace*{1.5mm}
	\item\label{item:sol2}
	$\int_\mathit{\Omega} \chi^\top\bm{A} \, {\rm d}\xi\in C([0,T];\mathbb{R}^m)\cap
	H_{\rm loc}^1((0,T];\mathbb{R}^m)$
	and $\int_\mathit{\Omega} \chi^\top\bm{A}(0) \, {\rm d}\xi=\int_\mathit{\Omega} \chi^\top\bm{A}_0 \, {\rm d}\xi$,\vspace*{1.5mm}
	\item\label{item:sol4} $\bm{A}\in L^2([0,T];X_0(\curl,\mathit{\Omega},\mathit{\Omega}_C))$ and $\bm{i}\in L_{\rm loc}^2((0,T];\mathbb{R}^m)$,
	\vspace*{1.5mm}
	\item\label{item:sol7} for all $\bm{F}\in X_0(\curl,\mathit{\Omega},\mathit{\Omega}_C)$ and almost all $t\in[0,T]$,
	\begin{equation}
	\arraycolsep=2pt
	\begin{array}{rcl}
	\displaystyle{\tfrac{\rm d }{{\rm d} t}\!\int_\mathit{\Omega} \!\sigma \bm{A}(t)\cdot \bm{F}\, {\rm d}\xi
		+\!\!	\int_\mathit{\Omega} \!\nu(\cdot,\|\nabla\!\times\! \bm{A}(t)\|_2)(\nabla\!\times\! \bm{A}(t))
		\cdot(\nabla\!\times\! \bm{F})\, {\rm d}\xi} & \!= &\!\!
	\displaystyle{\int_\mathit{\Omega} \!\chi\, \bm{i}(t)\cdot \bm{F}\,  {\rm d}\xi,} \\[2mm]
	\displaystyle{\tfrac{\rm d }{{\rm d} t}\!\int_\mathit{\Omega} \chi^\top\bm{A}(t)\, {\rm d}\xi+R\,\bm{i}(t) }& = &\! \bm{v}(t).
	\end{array}
	\label{eq:weak}
	\end{equation}
\end{enumerate}
It has been is proven in \cite[Theorem~4.4]{ChReSt21} that, under Assumptions~\ref{ass:omega}--\ref{ass:init}, the solution the coupled MQS system~\eqref{eq:MQS} is unique. Existence and some additional regularity properties of the solution have been established in \cite[Theorem~7.1]{ChReSt21}.
In particular, it has been shown there that for almost all $t\in[0,T]$,
\begin{equation}\label{eq:MQSsol}
\arraycolsep=2pt
\begin{array}{rcl}
\tfrac{\rm d}{{\rm d}t}\left(\sigma\bm{A}(t)\right) + \nabla \times \left(\nu(\cdot,\|\nabla \times \bm{A}(t)\|_2) \nabla \times \bm{A}(t)\right)  & =  & \chi\, \bm{i}(t) , \\[1mm]
\tfrac{\rm d }{{\rm d} t}\displaystyle{\int_\mathit{\Omega} \chi^\top\bm{A}(t)\, {\rm d}\xi + R\,\bm{i}(t)} & = & \bm{v}(t),
\end{array}
\end{equation}
and if, additionally, $\bm{A}_0\in X_0(\curl,\mathit{\Omega},\mathit{\Omega}_C)$, then the solution of \eqref{eq:MQS} fulfills
\begin{align*}
\nabla\times\bm{A}\in&\, L^\infty([0,T];X(\mathit{\Omega},\mathit{\Omega}_C)),&
\sigma\bm{A}\in&\, H^1([0,T];X(\mathit{\Omega},\mathit{\Omega}_C)),\\
\int_\mathit{\Omega} \chi^\top\bm{A} \, {\rm d}\xi\in&\, H^1([0,T];\R^m),&
\bm{i}\in&\, L^2([0,T];\R^m).
\end{align*}

\begin{remark}\label{rem:solinf}
	The above solution concept can easily be extended to the positive real axis {$\mathbb{R}_{\geq 0}$}
	by saying that, for $\bm{v}\in L_{\rm loc}^2({\mathbb{R}_{\geq 0}};\mathbb{R}^m)$, $(\bm{A},\bm{i})$ with $\bm{A}:\overline{\mathit{\Omega}}\times {\mathbb{R}_{\geq 0}}\to\mathbb{R}^3$ and
	$\bm{i}:{\mathbb{R}_{\geq 0}}\to\mathbb{R}^m$ is a~solution of the coupled MQS system \eqref{eq:MQS} if for any $T>0$, the restriction of $(\bm{A},\bm{i})$ to $[0,T]$ is a~solution of \eqref{eq:MQS}. In this case, existence
	and uniqueness of solutions immediately follow from the results on finite time intervals, and the above regularity results can be adapted straightforwardly.
\end{remark}

The coupled MQS model \eqref{eq:MQS} can be regarded as an abstract differential-algebraic {control} system
\begin{equation}
\arraycolsep=2pt
\begin{array}{rcl}
\tfrac{\rm d}{{\rm d}t}\cEl x(t)&=&\cAl(x(t))+\cBl u(t),\quad \cEl x(0)=\cEl x_0,\\
y(t)&=& \cCl x(t)
\end{array}
\label{eq:inf_sys}
\end{equation}
with the input $u(t)=\bm{v}(t)$,
the state $x(t)=(\bm{A}(t),\bm{i}(t))$,
the output $y(t)=\bm{i}(t)$,
and the initial value $x_0=(\bm{A}_0,0)$.
Further, the linear operators $\cEl$, $\cBl$, $\cCl$ and the nonlinear operator $\cAl$ are given by
\begin{subequations}\label{eq:EABCop}
	{
		\begin{align}
		\cEl:\; &&X(\mathit{\Omega},\mathit{\Omega}_C)\times \mathbb{R}^m&\,\to\, X(\mathit{\Omega},\mathit{\Omega}_C) \times \mathbb{R}^m,\\
		&&(\bm{A},\bm{i})&\,\mapsto\,\Bigl(\sigma\bm{A},\int_{\mathit{\Omega}} \chi^T\bm{A} \,{\rm d}\xi\Bigr),\nonumber\\[2.5mm]
		\cAl:\; && \hspace*{-10mm} X_0(\curl,\mathit{\Omega},\mathit{\Omega}_C) \times \mathbb{R}^m &\, \to\, X_0(\curl,\mathit{\Omega},\mathit{\Omega}_C)'\times \mathbb{R}^m,\\
		&&(\bm{A},\bm{i})&\,\mapsto\, \bigl(-\cAl_{11}^{}(\bm{A})+\chi\,\bm{i},-R\,\bm{i}\bigr),\qquad\qquad\nonumber\\[2.5mm]
		\cBl:\; && \mathbb{R}^m&\,\to\, X_0(\curl,\mathit{\Omega},\mathit{\Omega}_C)' \times \mathbb{R}^m,\\
		&&\bm{v}&\,\mapsto\,(0,\bm{v}),\nonumber\\[2.5mm]
		\cCl:\; && X(\mathit{\Omega},\mathit{\Omega}_C) \times\mathbb{R}^m&\,\to\, \mathbb{R}^m,\\
		&&(\bm{A},\bm{i})&\,\mapsto\,(0,\bm{i}), \nonumber
		\end{align}}
\end{subequations}
where
\begin{align*}
\qquad\cAl_{11}:\;&& X_0(\curl,\mathit{\Omega},\mathit{\Omega}_C) &\,\to \,X_0(\curl,\mathit{\Omega},\mathit{\Omega}_C)',\\
&&\bm{A}&\,\mapsto\,\displaystyle{\biggl(\bm{F}\mapsto
	\int_\mathit{\Omega} \nu(\cdot,\|\nabla\times \bm{A}\|_2)(\nabla\times \bm{A})\cdot(\nabla\times \bm{F})\, {\rm d}\xi\biggr).}\nonumber 
\end{align*}
Hereby, the first ``block row'' in $\tfrac{\rm d}{{\rm d}t}\cEl x(t)=\cAl(x(t))+\cBl u(t)$ corresponds to \eqref{eq:MQS1} and also includes the boundary condition \eqref{eq:MQSic1}, as this is included in the domain of $\cAl_{11}$. The second ``block row'' expresses the coupling relation \eqref{eq:MQScoupl}. The initial condition $\cEl x(0)=\cEl x_0$ comprises \eqref{eq:MQSic1} and \eqref{eq:MQSic2}, and the output equation \mbox{$y(t)=\cCl x(t)$} states that the output $y$ is given by the current $\bm{i}$.

\subsection{Magnetic energy}
\label{ssec:energ}

The magnetic energy plays an~important role in our forthcoming discussions.
For the magnetic reluctivity $\nu$ as in Assumption~\ref{ass:material}\,\ref{ass:material2}), consider the~function \mbox{$\vartheta:\mathit{\Omega}\times\mathbb{R}_{\ge0}\to\mathbb{R}_{\ge0}$} given by
\begin{equation}
\vartheta(\xi,\varrho)=\frac{1}{2}\int_0^\varrho \nu(\xi,\sqrt{\zeta})\, {\rm d}\zeta = \int_0^{\sqrt{\varrho}} \nu(\xi,\zeta)\zeta\,{\rm d}\zeta.
\label{eq:gamma}
\end{equation}
Further, define a~functional
\begin{equation}\label{eq:varphiA}
\begin{aligned}
E:&&X(\mathit{\Omega},\mathit{\Omega}_C)\to&\, \mathbb{R} \cup\{\infty\},\\
&& \bm{A}\mapsto&\,\begin{cases}\displaystyle{\int_{\mathit{\Omega}} \vartheta\bigl(\xi,\|\nabla\times \bm{A}(\xi )\|_2^2\bigr)\,{\rm d}\xi}\; & \text{if } \bm{A}{}\in X_0(\curl,\mathit{\Omega},\mathit{\Omega}_C),\\
\infty\; & \text{else}.\end{cases}
\end{aligned}\end{equation}
Given a~magnetic vector potential $\bm{A}(t)\in X_0(\curl,\mathit{\Omega},\mathit{\Omega}_C)$, the scalar function \linebreak \mbox{$\xi\mapsto \vartheta\bigl(\xi,\|\nabla\times \bm{A}(\xi,t)\|_2^2\bigr)$} is the {\em magnetic energy density}, and $E(\bm{A}(t))$ is the {\em magnetic energy}.
Many properties of the magnetic energy are collected in \cite[Proposition~5.2]{ChReSt21}. In particular, it has been shown there that $E$ is convex, and for all
$\bm{A}\in X_0(\curl,\mathit{\Omega},\mathit{\Omega}_C)$, it fulfills the estimates
\begin{equation} \label{eq:equivalence}
\frac{m_\nu}{2} \| \nabla \times \bm{A}\|_{L^2(\mathit{\Omega};\mathbb{R}^3)}^2 \leq E(\bm{A} ) \leq \frac{L_\nu}{2} \| \nabla \times \bm{A}\|_{L^2(\mathit{\Omega};\mathbb{R}^3)}^2,
\end{equation}
where $m_\nu,L_\nu>0$ are the constants as in Assumption~\ref{ass:material}\,\ref{ass:material2}).

\section{Passivity}
\label{sec:passivity}

In this section, we investigate passivity of the coupled MQS system~\eqref{eq:MQS}.
Passive systems form a~special class of dissipative dynamical systems which have extensively been studied in \cite{HillM80,Will72}. They are of particular interest in circuit si\-mu\-lation \cite{AndeV73} and controller design \cite{BroLME07}. Roughly speaking, a~system is passive if it does not generate energy or, equivalently, the energy is dissipated. Mathematically, passivity can be defined in terms of a~storage function. An~important property of passive systems is that an~interconnection of passive subsystems often provides a~new passive system.

Consider a~general abstract differential-algeb\-raic control system
\begin{equation}
\arraycolsep=2pt
\begin{array}{rcl}
\tfrac{\rm d}{{\rm d}t}\cEl(x(t))&=&\cAl(x(t))+\cBl(u(t)),\\
y(t)&=& \cCl(x(t))
\end{array}
\label{eq:inf_sys2}
\end{equation}
with (possibly nonlinear) operators $\cEl:X\to Z$, $\cAl:D(\cAl)\subset X\to Z$, $\cBl: U\to Z$ and
$\cCl: X\to U'$ for some Hilbert spaces\, $X$, $Z$ and  $U$.
The input $u\in L^2([0,T];U)$ is called {\em admissible with the initial condition} $\cEl(x(0))=\cEl(x_0)$, if \eqref{eq:inf_sys2} has a~solution $x:{[0,T]}\to X$ with $\cEl(x(0))=\cEl(x_0)$ and $y\in L^2([0,T];U')$. Similarly to the finite-dimensional case \cite{Will72}, we define passivity for the infinite-dimensional system \eqref{eq:inf_sys2} as follows.

\begin{definition}[Passivity]
	A~function $\mathcal{S}: {X}\to\mathbb{R}_{\ge0}\cup\{\infty\}$ is called a~{\em storage function  for passivity} of system \eqref{eq:inf_sys2}, if
		for all $T>0$, $x_0\in{X}$ with $\mathcal{S}(x_0)<\infty$ and all inputs $u\in L^2([0,T]; U)$ admissible with the initial condition
		$\cEl(x(0))=\cEl(x_0)$, the following conditions are fulfilled:
		\begin{enumerate}[\rm a)]
			\item $t\mapsto \mathcal{S}(x(t))$ is continuous as a function from $[0,T]$ to $\mathbb{R}_{\ge0}\cup\{\infty\}$;
			\item for all $0\leq t_0\leq t_1 \leq T$, it holds
			the {\em dissipation inequality}
			\begin{equation}
			\mathcal{S}(x(t_1))-\mathcal{S}(x(t_0))\leq \int_{t_0}^{t_1}\langle u(\tau), y(\tau)\rangle_{2}\,{\rm d}\tau,
			\label{eq:pass}
			\end{equation}
		\end{enumerate}
		where $\langle\cdot,\cdot\rangle_2$ stands for the standard Euclidean inner product in $\R^m$.
		System \eqref{eq:inf_sys2} is called {\em passive}, if there exists a storage function for passivity.
\end{definition}

The dissipation inequality \eqref{eq:pass} typically has the interpretation of an~energy ba\-lan\-ce. Namely, $\mathcal{S}(x(t))$ expresses the energy of the state $x(t)$, whereas the energy extracted from the system is given by $\int_{t_0}^{t_1} \langle u(\tau), y(\tau)\rangle_{2}\,{\rm d}\tau$. The nonnegative term
\[\int_{t_0}^{t_1} \langle u(\tau), y(\tau)\rangle_{2}\,{\rm d}\tau-\mathcal{S}(x(t_1))+\mathcal{S}(x(t_0))\]
is the energy which is dissipated by the system on the time interval $[t_0,t_1]$.

\newpage
\begin{remark}\
	\begin{enumerate}[\rm a)]
		\item Assume that the initial value $x_0$ fulfills $\mathcal{S}(x_0)<\infty$ and $u\in L^2([0,T];U)$ is admissible with the initial condition $\cEl(x(0))=\cEl(x_0)$. Then it immediately follows from the dissipation inequality \eqref{eq:pass} that $\mathcal{S}(x(t))<\infty$ for all $t\in[0,T]$.
		\item If, additionally,
		$\mathcal{S}(0)=0$, then for system \eqref{eq:inf_sys2} initialized with $\cEl(x(0))=\cEl(0)$ and for all inputs $u\in L^2([0,T]; U)$ admissible with this initial condition, the dissipation inequality \eqref{eq:pass} implies
		\[\mathcal{S}(x(t))\leq \int_{0}^{t} \langle u(\tau), y(\tau)\rangle_{2}\,{\rm d}\tau\quad\text{ for all } t\in [0,T].\]
		In this case, by the nonnegativity of $\mathcal{S}$, we have
		\[0\leq \int_{0}^{t} \langle u(\tau), y(\tau)\rangle_{2}\,{\rm d}\tau.\]
		Systems with this property are called {\em input-output passive}. It has been shown in \textup{\cite{HillM80}} that reachable
		and stabilizable input-output passive finite-dimensional standard state space systems possess a~storage function. In particular, such systems are passive. Passivity of infinite-dimensional linear systems has been studied in
		\textup{\cite{ReiJ08,Staf02}}.
	\end{enumerate}
\end{remark}

We now return to the coupled MQS system \eqref{eq:MQS}.
We have seen in Section~\textup{\ref{sec:solution}} that it can be written as an~abstract differential-algebraic control system \eqref{eq:inf_sys} with the operators as in \eqref{eq:EABCop} and, in particular, \mbox{$U=U'=\R^m$}. Then the existence result in \textup{\cite[Theorem~7.1]{ChReSt21}}
consequences that for \mbox{$\bm{A}_0\in X_0(\curl,\mathit{\Omega},\mathit{\Omega}_C)$},
any input \mbox{$u=\bm{v}\in L^2([0,T];\R^m)$} is admissible with the initial condition
$$
\cEl (\bm{A}(0),\bm{i}(0))=\cEl (\bm{A}_0,0).
$$

Next, we show that the function $\mathcal{S}_{\rm MQS}:X(\mathit{\Omega},\mathit{\Omega}_C)\times\R^m\to \mathbb{R}_{\ge0}\cup\{\infty\}$
defined by the magnetic energy as in \eqref{eq:varphiA}, i.e.,
\begin{equation}
\mathcal{S}_{\rm MQS}(\bm{A}(t),\bm{i}(t))= E(\bm{A}(t))
\label{eq:storageMQS}
\end{equation}
is a~storage function for passivity of the coupled MQS system~\eqref{eq:MQS}.
A~formal consideration by invoking
the chain rule for $\vartheta$ being as in \eqref{eq:gamma}, and the integration by parts formula with the weak curl operator, see \cite[eq.~(2.1)]{ChReSt21},
yields that the solution of system~\eqref{eq:MQS} fulfills
\begin{eqnarray*}
	{\tfrac{\rm d}{{\rm d}t}\mathcal{S}_{\rm MQS}(\bm{A}(t),\bm{i}(t))} &=&
	\displaystyle{
		\int_\mathit{\Omega}\tfrac{\partial}{\partial\varrho}\vartheta \bigl(\cdot,\|\nabla\times\bm{A}(t)\|_2^2\bigr)
		\tfrac{{\rm d}}{{\rm d} t}\|\nabla\times \bm{A}(t)\|_2^2 \,{\rm d}\xi} \\[2mm]
	&=& \displaystyle{-\int_\mathit{\Omega} \bigl\|\tfrac{{\rm d}}{{\rm d} t} \sqrt{\sigma}\bm{A}(t)\bigr\|_2^2\,{\rm d}\xi
		-\langle\bm{i}(t), R\,\bm{i}(t)\rangle_{2}+\langle\bm{v}(t),\bm{i}(t)\rangle_{2}}.
\end{eqnarray*}
Integrating this equation on ${[t_0,t_1]}$ and using that $u=\bm{v}$, $y=\bm{i}$,
and $R$ is a~positive definite matrix, we obtain
\[\begin{aligned}
\mathcal{S}_{\rm MQS}&(\bm{A}({t_1}),\bm{i}({t_1}))-\mathcal{S}_{\rm MQS}(\bm{A}{(t_0)},\bm{i}{(t_0)})\\=\;&-\int_{{t_0}}^{{t_1}}\int_\mathit{\Omega}\bigl\| \tfrac{{\rm d}}{{\rm d} \tau} \sqrt{\sigma}\bm{A}(\tau)\bigr\|_2^2\,{\rm d}\xi \,{\rm d}\tau-\int_{{t_0}}^{{t_1}}\langle\bm{i}(\tau),R\,\bm{i}(\tau)\rangle_{2}\,{\rm d}\tau+\int_{{t_0}}^{{t_1}} \langle\bm{v}(\tau),\bm{i}(\tau)\rangle_{2}\,{\rm d}\tau\\
\leq\;&\int_{{t_0}}^{{t_1}} \langle u(\tau), y(\tau)\rangle_{2}\,{\rm d}\tau.
\end{aligned}\]
In particular, the non-negative expression
$$
\int_{{t_0}}^{{t_1}}\int_\mathit{\Omega} \bigl\|\tfrac{{\rm d}}{{\rm d} \tau} \sqrt{\sigma}\bm{A}(\tau)\bigr\|_2^2\,{\rm d}\xi \,{\rm d}\tau+\int_{{t_0}}^{{t_1}} \langle\bm{i}(\tau), R\,\bm{i}^{}(\tau)\rangle_{2}\,{\rm d}\tau \geq 0
$$
stands for the energy dissipated by the system on the time interval ${[t_0,t_1]}$. Next, we show that a~rigorous analysis indeed leads to the above dissipation inequality.

\begin{theorem}[Energy balance for the coupled MQS system]\label{thm:energybalance}
	Assume that \mbox{$\mathit{\Omega}\subset\mathbb{R}^3$} with a~subdomain $\mathit{\Omega}_C$ satisfies Assumption~\textup{\ref{ass:omega}}. Further, let
	Assumptions~\textup{\ref{ass:material}} and~\textup{\ref{ass:init}} be fulfilled,
	$T>0$, $\bm{v}\in L^2([0,T];\mathbb{R}^m)$, and
	let $(\bm{A},\bm{i})$ be a~solution of the coupled MQS system \eqref{eq:MQS}.
	Then the magnetic energy function $E$ as defined in \eqref{eq:varphiA} has the following properties:
	\begin{align}
	\bigl(t\mapsto  E(\bm{A}(t))\bigr)\in&\, L^1({[0,T]})\cap W^{1,1}_{\rm loc}((0,T]),\label{eq:MQSBarbsyspass1} \\
	\bigl(t\mapsto t\, E(\bm{A}(t))\bigr)\in&\, L^\infty([0,T]).\label{eq:MQSBarbsyspass2}
	\end{align}
	Further, for all $0<t_0\leq t_1\leq T$, it holds
	\begin{multline}
	E(\bm{A}(t_1)) - E(\bm{A}(t_0))
	=-\int_{t_0}^{t_1}\|\tfrac{\rm d}{{\rm d}\tau} \sqrt{\sigma}\bm{A}(\tau)\|^2_{L^2(\mathit{\Omega};\mathbb{R}^3)}{\rm d}\tau\\-\int_{t_0}^{t_1}\langle\bm{i}(\tau),R\,\bm{i}(\tau)\rangle_{2}\,{\rm d}\tau+\int_{t_0}^{t_1} \langle\bm{v}(\tau), \bm{i}(\tau)\rangle_{2}\,{\rm d}\tau.
	\label{eq:MQSpass}\end{multline}
	If, additionally, $\bm{A}_{0}\in X_0(\curl,\mathit{\Omega},\mathit{\Omega}_C)$, then
	\begin{equation}
	\left(t\mapsto  E(\bm{A}(t))\right)\in\, W^{1,1}([0,T]),
	\label{eq:EinW11}
	\end{equation}
	and the identity \eqref{eq:MQSpass} holds for all $0\leq t_0\leq t_1\leq T$.
\end{theorem}

\begin{proof}
	It has been shown in the proof of \cite[Theorem~7.1]{ChReSt21} that the coupled MQS system \eqref{eq:MQS} can be formulated as an~abstract differential-algebraic gradient system with
	a~subgradient of $E$.
	This system meets the assumptions of \cite[Corollary~6.4]{ChReSt21}, which leads to the existence of a~unique solution $(\bm{A},\bm{i})$
	satisfying \eqref{eq:MQSBarbsyspass1}, \eqref{eq:MQSBarbsyspass2} and~\eqref{eq:MQSpass}.
	If, further, $\bm{A}_{0}\in X_0(\curl,\mathit{\Omega},\mathit{\Omega}_C)$, then
	\eqref{eq:EinW11} is fulfilled,
	and the identity \eqref{eq:MQSpass} even holds for all $0\leq t_0\leq t_1\leq T$.
\end{proof}

As a~consequence of Theorem~\ref{thm:energybalance}, we obtain the following result
establishing the passivity of the coupled MQS system \eqref{eq:MQS}.

\begin{corollary}
	Under the assumptions of Theorem~\textup{\ref{thm:energybalance}}, the function $\mathcal{S}_{\rm MQS}$ as in~\eqref{eq:storageMQS} is a~storage function for passivity of the coupled MQS system \eqref{eq:MQS}, and thus, this system is passive.
\end{corollary}

\begin{proof}
	Since~the expression
	\[\int_{t_0}^{t_1}\|\tfrac{\rm d}{{\rm d}\tau} \sqrt{\sigma}\bm{A}(\tau)\|^2_{L^2(\mathit{\Omega};\mathbb{R}^3)}\,{\rm d}\tau+\int_{t_0}^{t_1} \langle\bm{i}(\tau), R\,\bm{i}(\tau)\rangle_{2}\,{\rm d}\tau\]
	is positive, Theorem~\ref{thm:energybalance} implies that  for all $0<t_0\leq t_1\leq T$,
	the solution $(\bm{A},\bm{i})$ of~\eqref{eq:MQS}
	fulfills the dissipation inequality
	\[\mathcal{S}_{\rm MQS}(\bm{A}(t_1),\bm{i}(t_1))-\mathcal{S}_{\rm MQS}(\bm{A}(t_0),\bm{i}(t_0))\leq \int_{t_0}^{t_1} \langle\bm{v}(\tau), \bm{i}(\tau)\rangle_{2}\,{\rm d}\tau.\]
	Thus,  $\mathcal{S}_{\rm MQS}$ is a~storage function for passivity of the coupled MQS system \eqref{eq:MQS} and this system is indeed passive.	
\end{proof}

\begin{remark}
	\!\!Theorem~\textup{\ref{thm:energybalance}} implies that for $\bm{v}\!\in\! L^2([0,T];\R^m)$ and \mbox{$\bm{A}_0\!\in\! X(\mathit{\Omega},\mathit{\Omega}_C)$}, we have $\mathcal{S}_{\rm MQS}(\bm{A}(t),\bm{i}(t))<\infty$ for all $t\in(0,T]$. In other words, the storage function takes finite values after an arbitrary short time, even if it is infinite at time zero.
\end{remark}

\section{The Coupled MQS System in Port-Hamiltonian Formulation}
\label{sec:pH}

Port-Hamiltonian systems have meanwhile become rather popular as a~modelling tool especially for coupled (multi-)physical systems \cite{JvdS14}. In the past few years, this theory has successfully been extended to (finite-dimensional) differential-algebraic systems \cite{MvdS18,MvdS19,vdS13,GeHaRe21}. The aim of this section is to show that the coupled MQS system~\eqref{eq:MQS} also fits into the framework of port-Hamiltonian systems.

We first introduce some basics of port-Hamiltonian systems which are heavily inspired by \cite{JvdS14}. Since infinite-dimensional port-Hamiltonian systems have been so far treated mainly from a~differential geometric rather than from a~functional analytic perspective \cite{MvdS04a,MvdS04b,MasvdS02}, the authors of this paper have a~certain sovereignty of definition. Note that, though a~functional analytic approach to infinite-dimensional port-Hamiltonian systems has been discussed in \cite{JacZwa12}, the theory presented therein, however, restricts to the very limited class of coupled systems of linear transport equations.

An important concept is the Dirac structure which describes the power preserving energy-routing of the system. Dirac structures on Hilbert spaces have been considered in \cite{BKvdSZ10}, and their structure has been analysed by using the theory of Krein spaces.

\begin{definition}[Dirac structure]\label{def-Dir}
	Let $X$ be a~Banach space. A~subspace \linebreak \mbox{$\mathcal D \subset X\times X'$} is called a~\emph{Dirac structure}, if for all $f\in X$, $e\in X'$, it holds
		\begin{equation*}
		(f,e)\in \mathcal D\;\Longleftrightarrow\; \left( \langle{f},{\hat{e}}\rangle+\langle{\hat{f}},{{e}}\rangle=0\,\text{ for all }\, (\hat{f},\hat{e})\in \mathcal D\right).
		\end{equation*}
		Hereby, $X$ is called the space of {\em flows}, whereas $X'$ is called the space of {\em efforts}.
\end{definition}

Another concept needed for port-Hamiltonian systems is the resistive relation which represents the internal energy dissipation of the system \cite[Section~2.4]{JvdS14}.

\begin{definition}[Resistive relation]\label{def-res}
	Let $X$ be a~Banach space. A relation \linebreak\mbox{$\mathcal R\subset X\times X'$} is called \emph{resistive}, if
		\[\scpr{f}{e}\leq0\, \text{ for all }\,(f,e)\in{\mathcal{R}}.\]
\end{definition}

Having defined the Dirac structure and the resistive relation, we are now ready to introduce port-Hamiltonian systems. The subsequent definition uses the concepts of {\em G\^{a}teaux differentiability} and {\em G\^{a}teaux derivative}, for which we refer to \cite[Definition~4.5]{Zeid86}.

\begin{definition}[Port-Hamiltonian system]\label{def-pH}
	Let $X_{\mathcal{S}}$, $X_{\mathcal{R}}$ and $X_{\mathcal{P}}$ be Banach spaces.
		A \emph{port-Hamiltonian system}  is a~triple $(\mathcal D,\mathcal{H},\mathcal R)$, where
		$$
		\mathcal D\subset (X_{\mathcal{S}}\times X_{\mathcal{R}}\times X_{\mathcal{P}})\times(X_{\mathcal{S}}'\times X_{\mathcal{R}}'\times X_{\mathcal{P}}')
		$$
		is a~Dirac structure, $\mathcal{H}:X_{\mathcal{S}}\to \R$ is Lipschitz continuous on bounded sets and G\^{a}teaux differentiable, and $\mathcal R\subset X_{\mathcal{R}}\times X_{\mathcal{R}}'$ is a~resistive relation.
		The \emph{dynamics} of the port-Hamiltonian system on an~interval $\mathbb{I}\subset\R$ are specified by the differential inclusions
		\begin{equation}\label{eq:pHdyn}
		\arraycolsep=2pt
		\begin{array}{rcl}
		\bigl(-\tfrac{\rm d }{{\rm d} t}x(t),f_{\mathcal{R}}(t),f_{\mathcal{P}}(t),{\rm D}\mathcal{H}(x(t)),e_{\mathcal{R}}(t),e_{\mathcal{P}}(t))& \in &\mathcal D, \\
		(f_{\mathcal{R}}(t),e_{\mathcal{R}}(t)\bigr)& \in& \mathcal R,  \quad t\in\mathbb{I},
		\end{array}
		\end{equation}
		where ${\rm D}\mathcal{H}:X_{\mathcal{S}}\to X_{\mathcal{S}}'$ is the G\^{a}teaux derivative of $\mathcal{H}$.
		The function $\mathcal{H}$ is called {\em Hamiltonian}, whereas, for some time $t$ in which the system is driven, $x(t)$ is called {\em state}. The spaces $X_{\mathcal{S}}\times X_{\mathcal{S}}'$, $X_{\mathcal{R}}\times X_{\mathcal{R}}'$ and $X_{\mathcal{P}}\times X_{\mathcal{P}}'$ are referred to as {\em energy storage port}, {\em resistive port} and {\em external port}, respectively.	
\end{definition}

\begin{remark}\
	\begin{enumerate}[\rm a)]
		\item Though the above definition of a~port-Hamiltonian system includes the case of infinite-dimensional Banach spaces and therefore has a~certain generality, there are quite a~lot of opportunities for even further generalizations.
		For instance, in a very general setting, a Dirac structure is defined
		on a manifold $\mathcal M$ by a certain subbundle of $\mathcal D\subset T\mathcal M\oplus T^*\mathcal M$, where $T\mathcal M$
		is the tangent bundle and $T^*\mathcal M$ is the co-tangent bundle of $\mathcal{M}$ \textup{\cite[Definition~2.2.1]{Cou90}}. In \textup{\cite{MvdS18,MvdS19}}, the energy storage port has been determined by so-called {\em Lagrange manifolds}, which generalize the above use of Hamiltonians.
		\item Since the Hamiltonian is assumed to be Lipschitz continuous on bounded sets and is G\^{a}teaux differentiable, \textup{\cite[Theorem~4.2]{ArKr18}} can be applied to obtain that for any interval $\mathbb{I}\subset\R$ and any $x\in W^{1,p}(\mathbb{I};X_{\mathcal{S}})$, it holds that $(t\mapsto \mathcal{H}(x(t))\in W^{1,p}(\mathbb{I})$, and the weak derivative fulfills the generalized chain rule
		\begin{equation}\tfrac{\rm d }{{\rm d} t}\mathcal{H}(x)=\langle {\tfrac{\rm d}{{\rm d} t}}x,
		{\rm D}\mathcal{H}(x)\rangle.\label{eq:chain}
		\end{equation}
		Combining this with the properties of the Dirac structure and resistive structure, we obtain from \eqref{eq:pHdyn} that for almost all $t\in\mathbb{I}$, it holds the energy balance
		\[\begin{aligned}
		0=\,&-\langle \tfrac{\rm d }{{\rm d} t} x(t),{\rm D}\mathcal{H}(x(t))\rangle+
		\langle f_{\mathcal{R}}(t),e_{\mathcal{R}}(t)\rangle+\langle f_{\mathcal{P}}(t),e_{\mathcal{P}}(t)\rangle\\
		=\,&-\tfrac{\rm d }{{\rm d} t}\mathcal{H}(x(t))+
		\langle f_{\mathcal{R}}(t),e_{\mathcal{R}}(t)\rangle+\langle f_{\mathcal{P}}(t),e_{\mathcal{P}}(t)\rangle\\
		\leq\,&-\tfrac{\rm d }{{\rm d} t}\mathcal{H}(x(t))+
		\langle f_{\mathcal{P}}(t),e_{\mathcal{P}}(t)\rangle.
		\end{aligned}\]
		and, hence, an~integration on $[t_0,t_1]\subset \mathbb{I}$ yields 
		\begin{equation}
		\mathcal{H}(x(t_1))-\mathcal{H}(x(t_0))\leq\,
		\int_{t_0}^{t_1}\langle f_{\mathcal{P}}(t),e_{\mathcal{P}}(t)\rangle\,{\rm d}t.\label{eq:HamEnerg}
		\end{equation}
		If, for instance, the flows at the external ports form the input, and the efforts at the external ports form the output (or vice-versa), then we indeed obtain the dissipation inequality \eqref{eq:pass}.
		\item A~generalization of the G\^{a}teaux derivative is given by the {\em subdifferential} as considered in \textup{\cite{Bar10}} from the perspective of nonlinear evolution equations. This approach is applicable to a~class of Hamiltonians which are further allowed to map to $\R_{\ge0}\cup\{\infty\}$, and typically results into a~subdifferential which is set-valued and only defined on some subset of $X_{\mathcal{P}}$.
		Under the additional assumption that $X_{\mathcal{P}}$ is a~Hilbert space, it is shown in \textup{\cite[Lemma~4.4]{Bar10}} that the generalized chain rule \eqref{eq:chain} also holds, if the G\^{a}teaux derivative is replaced by a~subdifferential. Since this is the essential ingredient used in \eqref{eq:HamEnerg}, the incorporation of subdifferentials is a~further possible generalization of our approach to infinite-dimensional port-Hamiltonian systems.
	\end{enumerate}
\end{remark}

We now show that the coupled MQS system \eqref{eq:MQS} admits a~formulation as a~port-Hamiltonian system. For this purpose, we introduce the function
\begin{align*}
\mathcal{H}_{\rm MQS}: X_0(\curl,\mathit{\Omega},\mathit{\Omega}_C)&\to\,\R,\\
\bm{A}&\,\mapsto\int_{\mathit{\Omega}} \vartheta\bigl(\xi,\|\nabla\times\bm{A}(\xi)\|_2^2\bigr)\,{\rm d}\xi,
\end{align*}
with the magnetic energy density $\vartheta$ defined in \eqref{eq:gamma}.
Then it follows from \cite[Proposition~5.2\,a)]{ChReSt21} that $\mathcal{H}_{\rm MQS}$ is Lipschitz continuous on bounded sets, whereas  \cite[Proposition~5.2\,c)]{ChReSt21} shows that $\mathcal{H}_{\rm MQS}$ is G\^{a}teaux differentiable, and the G\^{a}teaux derivative fulfills for all $\bm{A},\bm{F}\in X_0(\curl,\mathit{\Omega},\mathit{\Omega}_C)$,
\begin{equation}\label{eq:Hamder}
\langle\bm{F},{\rm D}\mathcal{H}_{\rm MQS}(\bm{A})\rangle  = \int_{\mathit{\Omega}} \nu (\cdot , \| \nabla\times \bm{A} \|_2) \, (\nabla\times \bm{A}) \cdot (\nabla\times \bm{F})\,{\rm d}\xi.
\end{equation}
Next, consider the Hilbert spaces
\begin{equation}\label{eq:Xdef}
\begin{aligned}
X_{\mathcal{S}}=&\, X_0(\curl,\mathit{\Omega},\mathit{\Omega}_C),\qquad
X_{\mathcal{R}}=X_0(\curl,\mathit{\Omega},\mathit{\Omega}_C)'\times\R^m,\qquad
X_{\mathcal{P}}=\R^m, \\
X\,=&\, X_0(\curl,\mathit{\Omega},\mathit{\Omega}_C)\times X_0(\curl,\mathit{\Omega},\mathit{\Omega}_C)'\times\R^m \times\R^m.
\end{aligned}
\end{equation}
Since Hilbert spaces are reflexive, and the dual space of $\R^m$ can be canonically identified with itself, we have
\[\begin{aligned}
\!\!X_{\mathcal{S}}'=&\,X_0(\curl,\mathit{\Omega},\mathit{\Omega}_C)', \qquad
X_{\mathcal{R}}'=X_0(\curl,\mathit{\Omega},\mathit{\Omega}_C)\times\R^m,\qquad
X_{\mathcal{P}}'=\R^m,\\
X'\,=&\,X_0(\curl,\mathit{\Omega},\mathit{\Omega}_C)'\times X_0(\curl,\mathit{\Omega},\mathit{\Omega}_C)\times\R^m \times\R^m.
\end{aligned}\]
We introduce the relations
\begin{align}
\label{eq:Ddef}\mathcal D_{\rm MQS}=&\Biggl\{\left(\begin{pmatrix}\bm{f}_{\bm{A}}\\\bm{J}_{\rm int}\\\bm{v}_{\rm w}\\\bm{i}\end{pmatrix},\begin{pmatrix}{\bm{J}}\\\bm{E}\\\bm{i}_{\rm w}\\\bm{v}\end{pmatrix}\right)\in X\times X'\;:\; \bm{f}_{\bm{A}}=\bm{E},\;\;\bm{J}_{\rm int}=-\bm{J}+\chi\, \bm{i}_{\rm w},\Biggr.\\[-5mm]
&\hspace*{48mm}\Biggl. \bm{v}_{\rm w}=-\int_\mathit{\Omega} \chi^\top\bm{E} \, {\rm d}\xi-\bm{v},\enskip
\bm{i}=\bm{i}_{\rm w}\Biggr\},\nonumber\\[3mm]
\label{eq:Rdef}{\mathcal R}_{\rm MQS}=&\left\{\left(\begin{pmatrix}\bm{J}_{\rm int}\\\bm{v}_{\rm w}\end{pmatrix},\begin{pmatrix}\bm{E}\\\bm{i}_{\rm w}\end{pmatrix}\right)\in X_{\mathcal{R}}\times X_{\mathcal{R}}'\;:\; \bm{J}_{\rm int}=-\sigma\bm{E}, \enskip
\bm{v}_{\rm w}=-R\,\bm{i}_{\rm w}\right\}.
\end{align}
It immediately follows from Assumption~\ref{ass:material}\,\ref{ass:material1}) and~\ref{ass:resistance}) that ${\mathcal R}_{\rm MQS}$ is a~resistive relation. To prove that ${\mathcal D}_{\rm MQS}$ is a~Dirac structure, we advance the following lemma, which is a~straightforward generalization of the well-known statement that the graphs of a~skew-symmetric matrices define Dirac structures \cite{MvdS18,GeHaRe21}.

\begin{lemma}\label{lem:skewDirac}
	Let $X$ be a~Banach space, and let $\mathcal{J}\in\mathcal{L}(X',X)$ be a skew-dual operator in the sense that it fulfills {$\scpr{\mathcal{J}v}{w}=-\scpr{\mathcal{J}w}{v}$} for all $v,w\in X'$. Then $\mathcal{D}=\{(\mathcal{J}e,e)\enskip:\enskip e\in X'\}$ is a~Dirac structure.
\end{lemma}

\begin{proof}
	Assume that $(f,e)\in\mathcal{D}$. Then $f=\mathcal{J}e$. This implies that for all $(\hat{f},\hat{e})\in\mathcal{D}$,
	\[\langle{f},{\hat{e}}\rangle+\langle{\hat{f}},{{e}}\rangle=\langle{\mathcal{J}e},{\hat{e}}\rangle+\langle{\mathcal{J}\hat{e}},{{e}}\rangle=-\langle{\mathcal{J}e},{\hat{e}}\rangle+\langle{\mathcal{J}\hat{e}},{{e}}\rangle=0.\]
	On the other hand, if $(f,e)\in X\times X'$ fulfills $\langle{f},{\hat{e}}\rangle+\langle{\hat{f}},{{e}}\rangle=0$ for all $(\hat{f},\hat{e})\in\mathcal{D}$, then we obtain for all $\hat{e}\in X'$ that
	\[0=\langle{f},{\hat{e}}\rangle+\langle{\mathcal{J}\hat{e}},{{e}}\rangle=\langle{f},{\hat{e}}\rangle-\langle{\mathcal{J}e},{\hat{e}}\rangle=\langle{f-\mathcal{J}e},{\hat{e}}\rangle.\]
	Hence, \cite[Corollary 6.17~(2)]{Alt16} (which is a~direct consequence of the famous Hahn-Banach theorem) leads to $f=\mathcal{J}e$, which in turn gives $(f,e)\in\mathcal{D}$.
\end{proof}

We can now easily verify that $\mathcal{D}_{\rm MQS}$ as in \eqref{eq:Ddef} is a~Dirac structure.

\begin{proposition}
	Assume that $\mathit{\Omega}\subset\mathbb{R}^3$ with a~subdomain $\mathit{\Omega}_C$ satisfies Assumption~\textup{\ref{ass:omega}}, and let $\chi:\mathit{\Omega}\to\mathbb{R}^{3\times m}$ satisfy Assumption~\textup{\ref{ass:init}\,\ref{ass:winding})}. Then $\mathcal{D}_{\rm MQS}$ as in \eqref{eq:Ddef} is a~Dirac structure.
\end{proposition}

\begin{proof}
	For $X$ as in \eqref{eq:Xdef}, consider the bounded operator $\mathcal{J}_{\rm MQS}:X'\to X$ with
	\[\mathcal{J}_{\rm MQS}\begin{pmatrix}\bm{J}\\\bm{E}\\\bm{i}_{\rm w}\\\bm{v}\end{pmatrix}= \begin{pmatrix}\bm{E}\\-\bm{J}+\chi\, \bm{i}_{\rm w}\\-\int_\mathit{\Omega} \chi^\top\bm{E} \, {\rm d}\xi-\bm{v}\\\bm{i}_{\rm w}\end{pmatrix}.
	\]
	Then it can be immediately seen that $\mathcal{J}_{\rm MQS}$ is skew-dual, whence Lemma~\ref{lem:skewDirac} implies that $\mathcal{D}_{\rm MQS}$ is a~Dirac structure.
\end{proof}

Next, we show that the coupled MQS system \eqref{eq:MQS} represents the dynamics of the port-Hamiltonian system $(\mathcal{D}_{\rm MQS},\mathcal{H}_{\rm MQS},\mathcal{R}_{\rm MQS})$ in some sense. Denote
the state~$x(t)$ by $\bm{A}(t)$ and the flow and effort at the external port by {$\bm{i}(t)$} and {$\bm{v}(t)$}, respectively. Then
\begin{equation}
\begin{aligned}
\left(-\tfrac{\rm d }{{\rm d} t} \bm{A}(t),\begin{pmatrix}\bm{J}_{\rm int}(t)\\\bm{v}_{\rm w}(t)\end{pmatrix},{\bm{i}}(t),{\rm D}\mathcal{H}_{\rm MQS}(\bm{A}(t)),\begin{pmatrix}\bm{E}(t)\\\bm{i}_{\rm w}(t)\end{pmatrix},{\bm{v}}(t)\right)&\in{\mathcal D}_{\rm MQS},\\
\left(\begin{pmatrix}\bm{J}_{\rm int}(t)\\\bm{v}_{\rm w}(t)\end{pmatrix},\begin{pmatrix}\bm{E}(t)\\\bm{i}_{\rm w}(t)\end{pmatrix}\right)&\in{\mathcal R}_{\rm MQS}.
\end{aligned}\label{eq:MQSph}\end{equation}
By using the definition of $\mathcal{D}_{\rm MQS}$ in \eqref{eq:Ddef} and $\mathcal{R}_{\rm MQS}$ in \eqref{eq:Rdef}, we obtain
\begin{align*}
-\tfrac{\rm d }{{\rm d} t} \bm{A}(t)=&\, \bm{E}(t),&
-\sigma \bm{E}(t)=&\,-{\rm D}\mathcal{H}_{\rm MQS}(\bm{A}(t))+\chi\, \bm{i}_{\rm w}(t),\\
-R\,\bm{i}_{\rm w}(t)=&\,-\int_\mathit{\Omega} \chi^\top\bm{E}(t) \, {\rm d}\xi-\bm{v}(t),&
\bm{i}(t)=&\,\bm{i}_{\rm w}(t),
\end{align*}
and, thus,
\[\begin{aligned}
\tfrac{\rm d }{{\rm d} t} \sigma \bm{A}(t)=&\,-{\rm D}\mathcal{H}_{\rm MQS}(\bm{A}(t))+\chi\, \bm{i}(t),\\
\tfrac{\rm d }{{\rm d} t}\int_\mathit{\Omega} \chi^\top\bm{A}(t) \, {\rm d}\xi=&\,-R\,\bm{i}(t)+\bm{v}(t).
\end{aligned}\]
Since the first equation takes place in $X_0(\curl,\mathit{\Omega},\mathit{\Omega}_C)'$ and ${\rm D}\mathcal{H}_{\rm MQS}$ reads as in \eqref{eq:Hamder}, we obtain the weak formulation \eqref{eq:weak} of the coupled MQS system~\eqref{eq:MQS}.
Moreover, the external port is formed by the voltage $\bm{v}$ and the current $\bm{i}$ at the conductive contacts, which are, respectively, the input and the output of \eqref{eq:MQS}.
The variable $\bm{J}_{\rm int}$ in \eqref{eq:MQSph} stands for the current density in $\mathit{\Omega}$ induced by the electromagnetic field, whereas~$\bm{E}$ is the electric field intensity. Further, $\bm{v}_{\rm w}$ is the part of the voltage at the winding which is caused by the resistive effect, and $\bm{i}_{\rm w}$ is the corresponding current.

\begin{remark}
	Note that the port-Hamiltonian model \eqref{eq:MQSph} requires weak diffe\-ren\-tiability with respect to time of the magnetic vector potential $\bm{A}$, whereas the solution concept in Section~\textup{\ref{sec:solution}} only requires $\sigma\bm{A}$ to be weakly differentiable.
\end{remark}

The port-Hamiltonian model \eqref{eq:MQSph} can at least formally be represented in a~compact form
\[\begin{aligned}\tfrac{\rm d }{{\rm d} t}\begin{bmatrix}I\;&0\;&0\;\\0&0&0\\0&0&0\end{bmatrix}\!\begin{bmatrix}\bm{A}(t)\\\bm{E}(t)\\\bm{i}(t)\end{bmatrix}=&
\begin{bmatrix}
0\enskip\; &-I&\;0\\
I\enskip\; &-\sigma&-\chi\\
0\enskip\; &\int_\mathit{\Omega} \chi^\top\cdot\, {\rm d}\xi\enskip\; &-R\;\end{bmatrix}\!\begin{bmatrix}{\rm D}\mathcal{H}_{\rm MQS}(\bm{A}(t))\\\bm{E}(t)\\\bm{i}(t)\end{bmatrix}+\begin{bmatrix}0\\0\\I\end{bmatrix}\bm{v}(t),\\
\bm{i}(t)=&\begin{bmatrix}0\,&\,0\,&\,I\end{bmatrix}\begin{bmatrix}{\rm D}\mathcal{H}_{\rm MQS}(\bm{A}(t))\\\bm{E}(t)\\\bm{i}(t)\end{bmatrix},
\end{aligned}\]
whose structure amazingly resembles the class of finite-dimensional port-Hamiltonian differential-algebraic systems studied in \cite{MehM19ppt,BeMeXuZw18}.

\section{Solution Estimates}
\label{sec:stability}

In this section, we study the quantitative properties of the coupled MQS system~\eqref{eq:MQS} by considering its input-to-output and input-to-state behavior. We
present 
estimates for the current, the magnetic vector potential and the magnetic energy upon the initial magnetic vector potential and the voltage. We further show that, under some additional assumption on the initial value, the free dynamics of the MQS system decay exponentially.
Note that our estimates also cover the case of infinite intervals. To this end, we refer to Remark~\ref{rem:solinf} for the definition and existence of solutions on the whole positive real axis $\R_{\ge0}$.

\subsection{Input-to-output behavior}
\label{sec:io}

First, we derive estimates for the current~$\bm{i}$, which is considered as the output of the coupled MQS system \eqref{eq:MQS},
by means of the initial value $\bm{A}_0$ and the voltage $\bm{v}$, which forms the input. In most cases, we will stick to the case
$\bm{A}_0\in X_0(\curl,\mathit{\Omega},\mathit{\Omega}_C)$. Note that by \cite[Theorem~7.1]{ChReSt21} (see also Section~\ref{sec:solution}), the solution $(\bm{A},\bm{i})$ has the property that $\bm{A}(t)\in X_0(\curl,\mathit{\Omega},\mathit{\Omega}_C)$
for almost all $t>0$
even if the weak curl of
$\bm{A}_0\in X(\mathit{\Omega},\mathit{\Omega}_C)$ is not in $L^2(\mathit{\Omega};\R^3)$.

Our forthcoming considerations
rely on the following auxiliary result, which is solely based on the energy balance \eqref{eq:MQSpass} developed in Theorem~\ref{thm:energybalance}.

\begin{lemma}\label{lem:est}
	Assume that $\mathit{\Omega}\subset\mathbb{R}^3$ with subdomain $\mathit{\Omega}_C$  satisfies Assumption~\textup{\ref{ass:omega}}. Further, let
	Assumptions~\textup{\ref{ass:material}} and~\textup{\ref{ass:init}} be fulfilled, and let
	$T\in\R_{\ge0}\cup\{\infty\}$ and \linebreak \mbox{$\bm{v}\in L^2_{\rm loc}([0,T);\mathbb{R}^m)$}. Moreover, let
	$E$ be the magnetic energy as defined in \eqref{eq:varphiA}, and let
	$(\bm{A},\bm{i})$ be the~solution of the coupled MQS system \eqref{eq:MQS}. Then
	for all $\varepsilon>0$ and all {$0<t_0\leq t_1<T$}, it holds
	\begin{align*}	
	E(\bm{A}(t_1))&-E(\bm{A}(t_0))\\
	&\leq -\int_{t_0}^{t_1}\|\tfrac{\rm d}{{\rm d}\tau} \sqrt{\sigma}\bm{A}(\tau)\|^2_{L^2(\mathit{\Omega};\mathbb{R}^3)}{\rm d}\tau-\Bigl(1-\frac{\varepsilon}{2}\Bigr)\int_{t_0}^{t_1}\|R^{1/2}\bm{i}(\tau)\|_2^2\,{\rm d}\tau\\
	&\qquad+\frac1{2\varepsilon} \int_{t_0}^{t_1} \|R^{-1/2}\bm{v}(\tau)\|_2^2
	\,{\rm d}\tau.
	\end{align*}
	If, additionally, $\bm{A}_{0}\in X_0(\curl,\mathit{\Omega},\mathit{\Omega}_C)$, then the above inequality holds even for all \mbox{$0\leq t_0\leq t_1< T$}.
\end{lemma}

\begin{proof}
	Assume that $0< t_0\leq t_1< T$, and let $(\bm{A},\bm{i})$ be a~solution of \eqref{eq:MQS}.
	A~combination of the Cauchy-Schwarz with Young's inequality \cite[p.~53]{Alt16} yields that for almost all $\tau\in[t_0,t_1]$,
	\begin{align*}
	\langle\bm{i}(\tau), \bm{v}(\tau)\rangle_{2} & \leq \|R^{1/2}\bm{i}(\tau)\|_2\, \|R^{-1/2}\bm{v}(\tau)\|_2 \\
	& \leq\frac{\varepsilon}2\, \|R^{1/2}\bm{i}(\tau)\|_2^2+\frac1{2\varepsilon}\, \|R^{-1/2}\bm{v}(\tau)\|_2^2.
	\end{align*}
	Then we obtain from the energy balance \eqref{eq:MQSpass} that
	\[\begin{aligned}	
	&E(\bm{A}(t_1))-E(\bm{A}(t_0))\\
	&\quad \leq-\int_{t_0}^{t_1}\|\tfrac{\rm d}{{\rm d}\tau} \sqrt{\sigma}\bm{A}(\tau)\|^2_{L^2(\mathit{\Omega};\mathbb{R}^3)}{\rm d}\tau-\int_{t_0}^{t_1}\|R^{1/2}\bm{i}(\tau)\|_2^2\,{\rm d}\tau\\
	&\qquad+\int_{t_0}^{t_1} \frac{\varepsilon}{2} \|R^{1/2}\bm{i}(\tau)\|_2^2+\frac1{2\varepsilon} \|R^{-1/2}\bm{v}(\tau)\|_2^2\,{\rm d}\tau\\
	&\quad =-\int_{t_0}^{t_1}\|\tfrac{\rm d}{{\rm d}\tau} \sqrt{\sigma}\bm{A}(\tau)\|^2_{L^2(\mathit{\Omega};\mathbb{R}^3)}{\rm d}\tau-\Bigl(1-\frac{\varepsilon}{2}\Bigr)\int_{t_0}^{t_1}\|R^{1/2}\bm{i}(\tau)\|_2^2\,{\rm d}\tau\\&\qquad+\frac1{2\varepsilon} \int_{t_0}^{t_1} \|R^{-1/2}\bm{v}(\tau)\|_2^2
	\,{\rm d}\tau.
	\end{aligned}\]
	As Theorem~\ref{thm:energybalance} states that \eqref{eq:MQSpass} holds for all $0\leq t_0\leq t_1< T$, if, additionally,  $\bm{A}_{0}\in X_0(\curl,\mathit{\Omega},\mathit{\Omega}_C)$, the above estimate is then as well fulfilled in the remaining case $t_0=0$.
\end{proof}

In the following theorem, we establish an~estimate for the $L^2$-norm of the output $y=\bm{i}$ in terms of the $L^2$-norm of the input $u=\bm{v}$. In particular, we show that \mbox{$\bm{i}\in L^2(\R_{\ge0};\R^m)$}, if $\bm{v}\in L^2(\R_{\ge0};\R^m)$ and $\bm{A}_{0}\in X_0(\curl,\mathit{\Omega},\mathit{\Omega}_C)$.

\begin{theorem}\label{th:outputest}
	Assume that $\mathit{\Omega}\subset\mathbb{R}^3$ with a~subdomain $\mathit{\Omega}_C$ satisfies Assumption~\textup{\ref{ass:omega}}. Further, let
	Assumptions~\textup{\ref{ass:material}} and~\textup{\ref{ass:init}} be fulfilled, and let $T\in\R_{\ge0}\cup\{\infty\}$, \mbox{$\bm{v}\in L^2([0,T);\mathbb{R}^m)$}, and
	$\bm{A}_{0}\in X_0(\curl,\mathit{\Omega},\mathit{\Omega}_C)$. Moreover,
	let $E$ be the magnetic ener\-gy as defined in \eqref{eq:varphiA}, and let $(\bm{A},\bm{i})$
	be a~solution of the coupled MQS system~\eqref{eq:MQS}. Then $\bm{i}\in L^2([0,T);\R^m)$, and for all $\varepsilon\in(0,{2})$,
	\begin{align}
	\!\!\|R^{1/2}\bm{i}\|_{L^2([0,T);\R^m)}^2\!&\leq{\tfrac{2}{{2-\varepsilon}}}\,E(\bm{A}_0)+
	{\tfrac1{{\varepsilon (2-\varepsilon)}}}\,\|R^{-1/2}\bm{v}\|_{L^2([0,T);\R^m)}^2,\label{eq:MQSoutputest1}\\
	\!\!\|R^{1/2}\bm{i}\|_{L^2([0,T);\R^m)}\!&\leq\sqrt{{\tfrac{{L_\nu}}{{2-\varepsilon}}}}\|\nabla\!\times\! \bm{A}_0\|_{L^2(\mathit{\Omega};\mathbb{R}^3)}
	\!+\!\tfrac1{\sqrt{{\varepsilon (2-\varepsilon)}}}\,\|R^{-1/2}\bm{v}\|_{L^2([0,T);\R^m)},
	\!
	\label{eq:MQSoutputest2}
	\end{align}
	where $L_\nu$ is the Lipschitz constant as in Assumption~\textup{\ref{ass:material}~\ref{ass:material2})}. If, moreover, the initial value is zero, i.e., $\bm{A}_{0}=0$, then
	\begin{equation}
	\|R^{1/2}\bm{i}\|_{L^2([0,T);\R^m)}\leq \|R^{-1/2}\bm{v}\|_{L^2([0,T);\R^m)}.
	\label{eq:MQSoutputest3}
	\end{equation}
\end{theorem}

\begin{proof}
	By \cite[Theorem~7.1]{ChReSt21}, we have $\bm{i}\in L^2([0,T);\R^m)$.
	Further, let $t\in[0,T)$ and $\varepsilon\in(0,2)$. Then using Lemma~\ref{lem:est}, we obtain
	\[\begin{aligned}
	0\leq\,E(\bm{A}(t))
	\leq&\,E(\bm{A}_0)
	-\int_{0}^{t}\|\tfrac{\rm d}{{\rm d}\tau} \sqrt{\sigma}\bm{A}(\tau)\|^2_{L^2(\mathit{\Omega};\mathbb{R}^3)}{\rm d}\tau\\&\qquad-\Bigl(1-\frac{\varepsilon}{2}\Bigr)\int_{{0}}^{{t}}\|R^{1/2}\bm{i}(\tau)\|_2^2\,{\rm d}\tau+\frac1{2\varepsilon} \int_{{0}}^{{t}} \|R^{-1/2}\bm{v}(\tau)\|_2^2
	\,{\rm d}\tau\\
	\leq&\, E(\bm{A}_0)
	-\Bigl(1-\frac{\varepsilon}{2}\Bigr)\int_{{0}}^{{t}}\|R^{1/2}\bm{i}(\tau)\|_2^2\,{\rm d}\tau+\frac1{2\varepsilon} \int_{{0}}^{{t}} \|R^{-1/2}\bm{v}(\tau)\|_2^2
	\,{\rm d}\tau.
	\end{aligned}\]
	Adding the expression $\bigl(1-\frac{\varepsilon}{2}\bigr)\int_{0}^{t}\|R^{1/2}\bm{i}(\tau)\|_2^2\,{\rm d}\tau$ to both sides, dividing by $1-\tfrac{\varepsilon}{2}$, and subsequently taking the limit for $t\nearrow T$, we obtain \eqref{eq:MQSoutputest1}. The estimate \eqref{eq:MQSoutputest2} immediately follows by a~combination of
	\eqref{eq:MQSoutputest1} with \eqref{eq:equivalence}.
	Finally, if $\bm{A}_0=0$, then \eqref{eq:MQSoutputest3} follows from \eqref{eq:MQSoutputest2} by setting $\varepsilon={1}$.
\end{proof}

\newpage
\begin{remark}\
	\begin{enumerate}[\rm a)]
		\item Let $\lambda_{\min}(R)$ be the minimal {eigenvalue} of $R$. Then, under the assumptions of Theorem~\textup{\ref{th:outputest}}, we obtain from  \eqref{eq:MQSoutputest1} and \eqref{eq:MQSoutputest2}
		that
		\begin{align*}
		\|\bm{i}\|_{L^2([0,T);\R^m)}^2&\leq{\tfrac{2}{(2-\varepsilon){\lambda_{\min}(R)}}}\,E(\bm{A}_0)+
		{\tfrac1{{\varepsilon (2-\varepsilon)}{\lambda_{\min}^2(R)}}}\,\|\bm{v}\|_{L^2([0,T);\R^m)}^2,\\
		\|\bm{i}\|_{L^2([0,T);\R^m)}&\leq{\sqrt{\tfrac{L_\nu}{(2-\varepsilon)\lambda_{\min}(R)}}}\,\|\nabla\times \bm{A}_0\|_{L^2(\mathit{\Omega};\mathbb{R}^3)}\\&\qquad\nonumber+{\tfrac1{\sqrt{\varepsilon (2-\varepsilon)}\, \lambda_{\min}(R)}}\,\|\bm{v}\|_{L^2([0,T);\R^m)}.
		\end{align*}
		If, further, $\bm{A}_0=0$, then \eqref{eq:MQSoutputest3} implies
		\begin{equation}
		\|\bm{i}\|_{L^2([0,T);\R^m)}\leq \|R^{-1}\|\,\|\bm{v}\|_{L^2([0,T);\R^m)}=\tfrac{{1}}{\lambda_{\min}(R)}\,\|\bm{v}\|_{L^2([0,T);\R^m)}.
		\label{eq:MQSoutputest3alt}
		\end{equation}
		Note that in the above estimates, we can further replace $\lambda_{\min}(R)$ by $R$, if $m=1$.
		
		\item In systems theory and model reduction, estimates of the $L^2$-norm of the output by means of the $L^2$-norm of the input play a~crucial role. In the linear and time-invariant case, such estimates can be obtained by using the so-called \mbox{{\em $\mathcal{H}_\infty$-norm}} of the transfer function.
		{In particular, it has been shown in \textup{\cite{KerS21}} that discretizing the linear coupled MQS system \eqref{eq:MQS} using the finite element method
			and regularizing the resulting system, one obtains a~differential-algebraic system whose transfer function has the $\mathcal{H}_\infty$-norm coinciding with $\|R^{-1}\|$ independent of the fineness of the discretization. In this case, we have
			$$\|\bm{i}\|_{L^2([0,T);\R^m)}\leq \|R^{-1}\|\,\|\bm{v}\|_{L^2([0,T);\R^m)},$$
			and, further, in the case where $m=1$, this estimate is sharp in the sense that there exists a~sequence of inputs $(\bm{v}_k)_k$ in $L^2(\R_{\geq 0})$ such that for the corresponding sequence of outputs $(\bm{i}_k)_{k}$ of the system with $\bm{A}_0=0$, the sequence of quotients $\|\bm{i}_k\|_{L^2(\R_{\geq 0})}/\|\bm{v}_k\|_{L^2(\R_{\geq 0})}$ tends to $R^{-1}$.
			Thus, the estimate \eqref{eq:MQSoutputest3alt} is an~extension of this result
			to the quasilinear infinite-dimensional case, and it is presumably sharp.}
	\end{enumerate}
\end{remark}

\subsection{Input-to-state behavior}
\label{sec:is}

Next, we present a~quantitative estimate for the magnetic energy $E(\bm{A}(t))$ in terms of the $L^2$-norm of the input $\bm{v}$, if the coupled MQS system \eqref{eq:MQS} is initialized with
$\bm{A}_{0}\in X_0(\curl,\mathit{\Omega},\mathit{\Omega}_C)$.

\begin{proposition}\label{prop:stateest}
	Assume that $\mathit{\Omega}\subset\mathbb{R}^3$ with a~subdomain $\mathit{\Omega}_C$ satisfies Assumption~\textup{\ref{ass:omega}}. Further, let
	Assumptions~\textup{\ref{ass:material}} and~\textup{\ref{ass:init}} be fulfilled, and let $T\in\R_{\ge0}\cup\{\infty\}$,
	$\bm{v}\in L^2_{\rm loc}([0,T);\mathbb{R}^m)$, and
	$\bm{A}_{0}\in X_0(\curl,\mathit{\Omega},\mathit{\Omega}_C)$.
	Moreover,
	let $E$ be the magnetic energy as defined in \eqref{eq:varphiA}, and let $(\bm{A},\bm{i})$
	be a~solution of the coupled MQS system \eqref{eq:MQS}. Then for all
	$0\leq t< T$,
	\begin{align}
	E(\bm{A}(t))&\leq E(\bm{A}_0)+\tfrac14 \,{\|R^{-1/2}\bm{v}\|^2_{L^2([0,t);\R^m)}},
	\label{eq:MQSstateest1}\\
	\|\nabla\!\times\! \bm{A}(t)\|_{L^2(\mathit{\Omega};\mathbb{R}^3)}&\leq\sqrt{\tfrac{L_\nu}{m_\nu}}\,\|\nabla\!\times\! \bm{A}_0\|_{L^2(\mathit{\Omega};\mathbb{R}^3)}+\tfrac1{\sqrt{2m_\nu}}\,\|R^{-1/2}\bm{v}\|_{L^2([0,t);\R^m)},\label{eq:MQSstateest2}
	\end{align}
	where $m_\nu,L_\nu$ are the monotonicity and Lipschitz constants as in Assumption~\textup{\ref{ass:material}\,\ref{ass:material2})}. In particular, if $\bm{v}\in
	L^2([0,T);\R^m)$, then $\nabla\times \bm{A}\in L^\infty([0,T);L^2(\mathit{\Omega};\mathbb{R}^3))$ with
	\[
	\|\nabla\!\times\! \bm{A}\|_{L^\infty([0,T);L^2(\mathit{\Omega};\mathbb{R}^3))}
	\leq \sqrt{\tfrac{L_\nu}{m_\nu}}\,\|\nabla\!\times\! \bm{A}_0\|_{L^2(\mathit{\Omega};\mathbb{R}^3)}+
	\tfrac1{\sqrt{2m_\nu}}\,\|R^{-1/2}\bm{v}\|_{L^2([0,T);\R^m)}.
	\]
\end{proposition}

\begin{proof}
	The estimate \eqref{eq:MQSstateest1} follows from Lemma~\ref{lem:est} by choosing $\varepsilon=2$. Further, by invoking \eqref{eq:equivalence}, we obtain from \eqref{eq:MQSstateest1} that
	\[\begin{aligned}
	\|\nabla\times \bm{A}(t)\|_{L^2(\mathit{\Omega};\mathbb{R}^3)}^2
	\leq&\,\tfrac{2}{m_\nu}\, E(\bm{A}(t))\\
	\leq&\,\tfrac{2}{m_\nu}\,E(\bm{A}_0)+\tfrac{1}{2m_\nu} {\|R^{-1/2}\bm{v}\|^2_{L^2([0,t);\R^m)}}\\
	\leq&\,\tfrac{L_\nu}{m_\nu}\|\nabla\times \bm{A}_0\|_{L^2(\mathit{\Omega};\mathbb{R}^3)}^2+
	\tfrac{1}{2m_\nu} {\|R^{-1/2}\bm{v}\|^2_{L^2([0,t);\R^m)}}\\
	\leq&\,{\Bigl(\sqrt{\tfrac{L_\nu}{m_\nu}}\,\|\nabla\times \bm{A}_0\|_{L^2(\mathit{\Omega};\mathbb{R}^3)}+
		\tfrac{1}{\sqrt{2m_\nu}}\, \|R^{-1/2}\bm{v}\|_{L^2([0,t);\R^m)}\Bigr)^2}.
	\end{aligned}\]
	Thus, \eqref{eq:MQSstateest2} follows by taking the square root.
\end{proof}

\begin{remark}
	Let $\lambda_{\min}(R)$ be the minimal {eigenvalue} of $R$. Then, under the assumptions of Proposition~\ref{prop:stateest},
	we obtain that
	\begin{align*}
	E(\bm{A}(t))&\,\leq E(\bm{A}_0)+\tfrac1{4\,{\lambda_{\min}(R)}} \,{\|\bm{v}(\tau)\|^2_{L^2([0,t);\R^m)}},
	\\
	\|\nabla\times \bm{A}(t)\|_{L^2(\mathit{\Omega};\mathbb{R}^3)}&\,\leq\sqrt{\tfrac{L_\nu}{m_\nu}}\,\|\nabla\times \bm{A}_0\|_{L^2(\mathit{\Omega};\mathbb{R}^3)}+\tfrac1{\sqrt{2m_\nu {\lambda_{\min}(R)}}}\,\|\bm{v}\|_{L^2([0,t);\R^m)},\\
	\|\nabla\times \bm{A}\|_{L^\infty([0,T);L^2(\mathit{\Omega};\mathbb{R}^3))}\!\!\!\!\!\!\!\!\!\!\!\!\!\!\!\!\!\!\!\!\!\!\\
	&\,\leq
	\sqrt{\tfrac{L_\nu}{m_\nu}}\,\|\nabla\times \bm{A}_0\|_{L^2(\mathit{\Omega};\mathbb{R}^3)}+\tfrac1{\sqrt{2m_\nu \lambda_{\min}(R)}}\,\|\bm{v}\|_{L^2([0,T);\R^m)}.
	\end{align*}
	We can again replace $\lambda_{\min}(R)$ by $R$ in the single-input single-output case, i.e., $m=1$.
\end{remark}

Finally, we analyze the dependence of $\|\bm{A}(t)\|_{L^2(\mathit{\Omega};\mathbb{R}^3)}$ upon the input
$\bm{v}$ and the initial value $\bm{A}_0$.
Hereby, we make essential use of \cite[Lemma~3.4]{ChReSt21} which states that there exists a~constant $L_C>0$ such that for all $\bm{A}\in X_0(\curl,\mathit{\Omega},\mathit{\Omega}_C)$,
\begin{equation}
\|\bm{A}\|_{L^2(\mathit{\Omega};\R^3)}^2\leq L_C\,\left(\|\bm{A}\|_{L^2(\mathit{\Omega_C};\R^3)}^2+\|\nabla\times \bm{A}\|_{L^2(\mathit{\Omega};\R^3)}^2\right).
\label{eq:curlest}
\end{equation}
As a~preliminary thought, we use the fact that for the electric conductivity $\sigma$
as in Assumption~\ref{ass:material}\,\ref{ass:material1}), one has
$\langle\sigma\bm{A}_1,\bm{A}_2\rangle_{L^2(\mathit{\Omega};\mathbb{R}^3)}=\sigma_C\langle\bm{A}_1,\bm{A}_2\rangle_{L^2(\mathit{\Omega}_C;\mathbb{R}^3)}$.
Then \eqref{eq:curlest} implies that
\begin{multline}
\langle\bm{A}_1,\bm{A}_2\rangle_{X_0(\curl,\mathit{\Omega},\mathit{\Omega}_C)}=
\langle\sigma\bm{A}_1,\bm{A}_2\rangle_{L^2(\mathit{\Omega};\mathbb{R}^3)}+\langle\nabla\times\bm{A}_1,\nabla\times\bm{A}_2\rangle_{L^2(\mathit{\Omega};\mathbb{R}^3)} \\
+\int_\mathit{\Omega} \chi^\top \bm{A}_1\,  {\rm d}\xi\cdot R^{-1} \int_\mathit{\Omega} \chi^\top \bm{A}_2\,  {\rm d}\xi\qquad\qquad\quad\quad
\label{eq:Xinner}
\end{multline}
defines an inner product in $X_0(\curl,\mathit{\Omega},\mathit{\Omega}_C)$ whose induced norm is equivalent to the standard norm in
$H_0(\curl,\mathit{\Omega})$.
Consider now the space
\begin{equation} X_0(\curl\!=\!0,\mathit{\Omega},\mathit{\Omega}_C) =\{\bm{A}\in X_0(\curl,\mathit{\Omega},\mathit{\Omega}_C)\enskip:\enskip\nabla\times\bm{A}=0\}
\label{eq:X0curl0}\end{equation}
which is a~closed subspace of both $X(\mathit{\Omega},\mathit{\Omega}_C)$ and $X_0(\curl,\mathit{\Omega},\mathit{\Omega}_C)$ with respect to the respective norms.

The following lemma is essential for our further analysis.

\begin{lemma}\label{lem:orthproj}
	Assume that $\mathit{\Omega}\subset\mathbb{R}^3$ with a~subdomain $\mathit{\Omega}_C$ satisfies Assumption~\textup{\ref{ass:omega}}. Let $P\in \mathcal{L}(X_0(\curl,\mathit{\Omega},\mathit{\Omega}_C))$ be the orthogonal projector onto $X_0(\curl\!=\!0,\mathit{\Omega},\mathit{\Omega}_C)$ with respect to the inner product \eqref{eq:Xinner}. Then the following statements hold:
	\begin{enumerate}[\rm a)]
		\item\label{lem:orthproja} The orthogonal projector $P$ can be uniquely extended to a~bounded projector $\widetilde{P}\in\mathcal{L}(X(\mathit{\Omega},\mathit{\Omega}_C))$, where $X(\mathit{\Omega},\mathit{\Omega}_C)$ is provided with the norm in $L^2(\mathit{\Omega};\mathbb{R}^3)$. The operator norm of $\widetilde{P}$ fulfills $\|\widetilde{P}\|\leq \sqrt{\tfrac{\gamma \,L_C}{\sigma_C}}$
		with $\sigma_C$ as in Assumption~\textup{\ref{ass:material}\,\ref{ass:material1})}, $L_C$ as in \eqref{eq:curlest}, and
		\begin{equation}\label{eq:cconst}
		\gamma={\sigma_C}+\|\chi\, R^{-1/2}\|^2_{L^2(\mathit{\Omega};\R^{3\times m})}.
		\end{equation}
		\item\label{lem:orthprojb} There exists a~constant ${L_1}>0$ such that for all $\bm{A}\in X_0(\curl,\mathit{\Omega},\mathit{\Omega}_C)$,
		\begin{equation}
		\|(I-P)\bm{A}\|_{L^2(\mathit{\Omega};\mathbb{R}^3)}\leq L_1\|\nabla\times \bm{A}\|_{L^2(\mathit{\Omega};\mathbb{R}^3)}.\label{eq:L1const}
		\end{equation}
		\item\label{lem:orthprojc} The operator
		\[\begin{aligned}
		\mathcal{T}:&&\!\!\im \widetilde{P}&\;\to \im\widetilde{P}^*,\\
		&&\!\!\bm{A}&\;\mapsto \widetilde{P}^*\Bigl(\sigma \bm{A}+\chi\, R^{-1}\int_\mathit{\Omega} \chi^\top \bm{A}\,  {\rm d}\xi\Bigr)
		\end{aligned}\]
		has a~bounded inverse. The operator norm of this inverse fulfills
		$\|\mathcal{T}^{-1}\|\leq \tfrac{L_C}{\sigma_C}$.
	\end{enumerate}
\end{lemma}

\begin{proof}
	\ref{lem:orthproja}) Let $\bm{A}\in X_0(\curl,\mathit{\Omega},\mathit{\Omega}_C)$. Then $\nabla\times (P\bm{A})=0$, and we obtain by using the definition of the inner product \eqref{eq:Xinner} that
	\[\begin{aligned}
	0 \,&\, = \langle P\bm{A},(I-P)\bm{A}\rangle_{X_0(\curl,\mathit{\Omega},\mathit{\Omega}_C)} \\
	&\,=\int_{\mathit{\Omega}}\sigma(P\bm{A})\cdot((I-P)\bm{A})\,{\rm d}\xi+
	\int_\mathit{\Omega} \chi^\top (P\bm{A})\,  {\rm d}\xi\cdot R^{-1}
	\int_\mathit{\Omega} \chi^\top ((I-P)\bm{A})\,  {\rm d}\xi.
	\end{aligned}\]
	This relation leads to
	\begin{align*}
	\|P\bm{A}&\|_{L^2(\mathit{\Omega};\R^3)}^2
	\overset{\eqref{eq:curlest}}{\leq}
	L_C\,\|P\bm{A}\|_{L^2(\mathit{\Omega}_C;\R^3)}^2
	= \tfrac{L_C}{\sigma_C}\,\int_{\mathit{\Omega}}\sigma\|P\bm{A}\|_2^2\,{\rm d}\xi\\
	\leq &\;\tfrac{L_C}{\sigma_C}\,\biggl(\int_{\mathit{\Omega}}\sigma\|P\bm{A}\|_2^2\,{\rm d}\xi
	{+\int_{\mathit{\Omega}}\!\sigma\|(I-P)\bm{A}\|_2^2\,{\rm d}\xi} 		\biggr.\\
	&+
		{2}\!\int_{\mathit{\Omega}}\sigma(P\bm{A})\cdot((I-P)\bm{A})\,{\rm d}\xi+{2}\!\int_\mathit{\Omega} \!\chi^\top (P\bm{A})\,  {\rm d}\xi\cdot R^{-1}\! \int_\mathit{\Omega} \!\chi^\top ((I-P)\bm{A})\,  {\rm d}\xi
	\\
	&\biggl.\;
	{+\Bigl\|  R^{-1/2}\int_\mathit{\Omega} \chi^\top (P\bm{A})\,  {\rm d}\xi\Bigr\|_2^2}
	+\Bigl\| R^{-1/2} \int_\mathit{\Omega} \chi^\top ((I-P)\bm{A})\,  {\rm d}\xi\Bigr\|_2^2\biggr)\\
	= &\;\tfrac{L_C}{\sigma_C}\,\biggl(\int_{\mathit{\Omega}_C}  {\sigma_C}\|\bm{A}\|_2^2\,{\rm d}\xi+
	\Bigl\| R^{-1/2} \int_\mathit{\Omega} \chi^\top \bm{A}\,  {\rm d}\xi\Bigr\|_2^2\biggr)
	\leq \tfrac{\gamma\,L_C}{\sigma_C}\,\|\bm{A}\|_{L^2(\mathit{\Omega};\R^3)}^2.
	\end{align*}
	Since $X_0(\curl,\mathit{\Omega},\mathit{\Omega}_C)$ is dense in $X(\mathit{\Omega},\mathit{\Omega}_C)$, we can make use of \cite[Theorem~E5.3]{Alt16} to see that the projector
	$P$ uniquely extends to an operator $\widetilde{P}\in\mathcal{L}(X(\mathit{\Omega},\mathit{\Omega}_C))$
	{with $\|\widetilde{P}\|^2\leq {\tfrac{\gamma\,L_C}{\sigma_C}}$}.
	As the operator $\widetilde{P}^2-\widetilde{P}\in\mathcal{L}(X(\mathit{\Omega},\mathit{\Omega}_C))$ vanishes on the dense subspace $X_0(\curl,\mathit{\Omega},\mathit{\Omega}_C)$, it has to vanish everywhere. Consequently, $\widetilde{P}$ is a~projector.
	
	\ref{lem:orthprojb}) Step 1: First, we show that the mapping
	\begin{align*}
	\Psi:\quad X_0(\curl,\mathit{\Omega},\mathit{\Omega}_C)&\,\to\,  L^2(\divg\!=\!0,\mathit{\Omega};\mathbb{R}^3),\\
	\bm{A}&\,\mapsto\, \nabla\times \bm{A}
	\end{align*}
	is surjective.
	Let $\bm{F}\in L^2(\divg\!=\!0,\mathit{\Omega};\mathbb{R}^3)$. Then, by \cite[Theorem~3.17]{ABDG98}, there exists some $\mathbf{C}\in H_0(\curl,\mathit{\Omega})$ such that $\bm{F}=\nabla\times \mathbf{C}$. By definition of $X(\mathit{\Omega},\mathit{\Omega}_C)$, we may consider an orthogonal decomposition {$\mathbf{C}=\bm{A}+\nabla {\psi}$}
	with
	$\bm{A}\in X(\mathit{\Omega},\mathit{\Omega}_C)$ and \mbox{$\psi\in {H_0^1}(\mathit{\Omega})$} which is constant on each boundary component of $\mathit{\Omega}_C$.
	Consequently, the tangential boundary trace of $\nabla\psi$ vanishes, whence this also holds for \mbox{$\bm{A}=\mathbf{C}-\nabla {\psi}$}. Then we obtain $\nabla\times \bm{A}=\nabla\times \mathbf{C}-\nabla\times \nabla\psi=\nabla\times \mathbf{C}=\bm{F}$
	and $\bm{A}\in X_0(\curl,\mathit{\Omega},\mathit{\Omega}_C)$, i.e., $\bm{F}\in\im \Psi$.
	
	Step 2: Next, we show that the restriction $\left.\Psi\right|_{\ker P}$ of $\Psi$ to $\ker P$ is bijective.
	{By definition of $P$, we have $\ker\Psi=\im P$, which is the orthogonal complement of $\ker P$ with respect to the inner product~\eqref{eq:Xinner}. Therefore, $\left.\Psi\right|_{\ker P}$ is injective.}
	To prove surjectivity, let $\bm{F}\in L^2(\divg\!=\!0,\mathit{\Omega};\mathbb{R}^3)$.
	Then, by Step~1, there exists some $\bm{A}\in X_0(\curl,\mathit{\Omega},\mathit{\Omega}_C)$ with $\bm{F}=\nabla\times \bm{A}$, and thus
	\[\nabla\times ((I-P)\bm{A})=\nabla\times ((I-P)\bm{A})+\nabla\times (P\bm{A})=\nabla\times \bm{A}=\bm{F}.\]
	
	Step 3: Finally, we show that there exists $L_1>0$ such that \eqref{eq:L1const} holds.
	We have seen in Step~2 that $\left.\Psi\right|_{\ker P}:\ker P\to L^2(\divg\!=\!0,\mathit{\Omega};\mathbb{R}^3)$ is bijective. It can be further shown that this mapping is bounded. Then the inverse mapping theorem \linebreak \cite[Theorem~7.8]{Alt16} yields that $\left.\Psi\right|_{\ker P}$ has a~bounded inverse, which implies that there exists a~constant $c_1>0$ such that for all $\bm{A}\in X_0(\curl,\mathit{\Omega},\mathit{\Omega}_C)$,
	$$
	\|(I-P)\bm{A}\|_{L^2(\mathit{\Omega};\R^3)}^2+\|\nabla\times ((I-P)\bm{A})\|_{L^2(\mathit{\Omega};\R^3)}^2\leq c_1 \|\nabla\times(I-P)\bm{A}\|_{L^2(\mathit{\Omega};\R^3)}^2.
	$$
	Then, clearly, $c_1>1$, and we obtain that \eqref{eq:L1const} holds with $L_1=c_1-1>0$.
	
	\ref{lem:orthprojc}) Let $\bm{A}\in X_0(\curl,\mathit{\Omega},\mathit{\Omega}_C)$. Then
	\allowdisplaybreaks
	\begin{align*}
	\langle \widetilde{P}\bm{A},&\,\mathcal{T}\widetilde{P}\bm{A}\rangle_{L^2(\mathit{\Omega};\mathbb{R}^3)}
	=\langle \widetilde{P}\bm{A},\mathcal{T}{P}\bm{A}\rangle_{L^2(\mathit{\Omega};\mathbb{R}^3)}\\
	&\;= \Bigl\langle\widetilde{P}\bm{A}, \widetilde{P}^*\bigl(\sigma P\bm{A}+\chi R^{-1}\int_\mathit{\Omega} \chi^\top ( P\bm{A})\,  {\rm d}\xi \bigr) \Bigr\rangle_{L^2(\mathit{\Omega};\mathbb{R}^3)}\\
	&\;= \Bigl\langle\widetilde{P}\widetilde{P}\bm{A}, \sigma P\bm{A}+\chi R^{-1}\int_\mathit{\Omega} \chi^\top ( P\bm{A})\,  {\rm d}\xi \Bigr\rangle_{L^2(\mathit{\Omega};\mathbb{R}^3)}\\
	&\;= \Bigl\langle P\bm{A}, \sigma P\bm{A}+\chi R^{-1}\int_\mathit{\Omega} \chi^\top ( {P}\bm{A})\,  {\rm d}\xi \Bigr\rangle_{L^2(\mathit{\Omega};\mathbb{R}^3)}\\
	&\;=\int_{\mathit{\Omega}}\sigma(P\bm{A})\cdot(P\bm{A})\,{\rm d}\xi+\int_\mathit{\Omega} \chi^\top (P\bm{A})\,  {\rm d}\xi\cdot R^{-1} \int_\mathit{\Omega} \chi^\top (P\bm{A})\,  {\rm d}\xi\\
	&\;\geq\int_{\mathit{\Omega}}\sigma(P\bm{A})\cdot(P\bm{A})\,{\rm d}\xi\\
	&\;=\sigma_C\Bigl(\|P\bm{A}\|_{L^2(\mathit{\Omega}_C;\mathbb{R}^3)}^2+\|\nabla\times (P\bm{A})\|_{L^2(\mathit{\Omega};\mathbb{R}^3)}^2\Bigr)\\
	&\;
	\overset{\eqref{eq:curlest}}{\geq}\,\tfrac{\sigma_C}{L_C}\|P\bm{A}\|_{L^2(\mathit{\Omega};\mathbb{R}^3)}^2.
	\end{align*}
	Since $X_0(\curl,\mathit{\Omega},\mathit{\Omega}_C)$ is dense in $X(\mathit{\Omega},\mathit{\Omega}_C)$, we have
	\[\langle \widetilde{P}\bm{A},\mathcal{T}\widetilde{P}\bm{A}\rangle_{L^2(\mathit{\Omega};\mathbb{R}^3)}\geq \tfrac{\sigma_C}{L_C}\|\widetilde{P}\bm{A}\|_{L^2(\mathit{\Omega};\mathbb{R}^3)}^2 \,\text{ for all }\,\bm{A}\in X(\mathit{\Omega},\mathit{\Omega}_C).\]
	Consequently, $\mathcal{T}$ has a~bounded inverse with $\|\mathcal{T}^{-1}\|\leq \frac{L_C}{\sigma_C}$.
\end{proof}

For the next result on the dependence of the $L^2$-norm of $\bm{A}(t)$ upon the input
and the initial value, we recall that we use the identification \eqref{eq:L23m} and the norm \eqref{eq:chinorm} on $L^2(\mathit{\Omega};\mathbb{R}^{3\times m})$.

\begin{theorem}\label{thm:stateest2}
	Assume that $\mathit{\Omega}\subset\mathbb{R}^3$ with a~subdomain $\mathit{\Omega}_C$
	satisfies Assumption~\textup{\ref{ass:omega}}. Further, let
	Assumptions~\textup{\ref{ass:material}} and~\textup{\ref{ass:init}} be fulfilled, and let $T\in\R_{\ge0}\cup\{\infty\}$, \mbox{$\bm{v}\in L^2_{\rm loc}([0,T);\mathbb{R}^m)$}, and
	$\bm{A}_{0}\in X_0(\curl,\mathit{\Omega},\mathit{\Omega}_C)$. Moreover,
	let $E$
	be the magnetic energy as defined in \eqref{eq:varphiA}, and let $(\bm{A},\bm{i})$
	be a~solution of the coupled MQS system \eqref{eq:MQS}. Let $m_\nu$ and $L_\nu$ be as in Assumption~\textup{\ref{ass:material}\,\ref{ass:material2})}, $L_C$ as in \eqref{eq:curlest}, $L_1$ as in Lemma~\textup{\ref{lem:orthproj}\,\ref{lem:orthprojb})}, and $\gamma$ as in \eqref{eq:cconst}.
	\begin{enumerate}[\rm a)]
		\item\label{thm:stateest2a} For all $0<t\leq T$,
		\begin{multline*}
		\|\bm{A}(t)\|_{L^2(\mathit{\Omega};\mathbb{R}^3)}
		\leq L_1\sqrt{\tfrac{L_\nu}{m_\nu}}\,\|\nabla\times \bm{A}_0\|_{L^2(\mathit{\Omega};\mathbb{R}^3)}+\tfrac{L_1}{\sqrt{2m_\nu}}\,\|R^{-1/2}\bm{v}\|_{L^2([0,t);\R^m)}\\+\sqrt{\tfrac{\gamma\,L_C}{\sigma_C}}\left(\!
		\|\bm{A}_0\|_{L^2(\mathit{\Omega};\mathbb{R}^3)}\!+\tfrac{L_C}{\sigma_C}  \|\chi R^{-1/2}\|_{L^2(\mathit{\Omega};\mathbb{R}^{3\times m})} \left\|\int_0^t \!\! R^{-1/2} \bm{v}(\tau){\rm d}\tau\right\|_2\right).
		\end{multline*}
		\item\label{thm:stateest2b}
		If $\chi\in L^2(\divg\!=\!0,\mathit{\Omega};\R^{3})^{1\times m}$, then
		\begin{multline*}
		\|\bm{A}(t)\|_{L^2(\mathit{\Omega};\mathbb{R}^3)}
		\leq\sqrt{\tfrac{\gamma\,L_C}{\sigma_C}}\;\|\bm{A}_0\|_{L^2(\mathit{\Omega};\mathbb{R}^3)}\nonumber
		\\+L_1\sqrt{\tfrac{L_\nu}{m_\nu}}\,\|\nabla\times \bm{A}_0\|_{L^2(\mathit{\Omega};\mathbb{R}^3)}
		+\tfrac{L_1}{\sqrt{2m_\nu}}\,\|R^{-1/2}\bm{v}\|_{L^2([0,t);\R^m)}.
		\end{multline*}
	\end{enumerate}
\end{theorem}

\begin{proof}
	\ref{thm:stateest2a})
	Let $(\bm{A},\bm{i})$ be a~solution of  \eqref{eq:MQS}, and let \mbox{$P\in \mathcal{L}(X_0(\curl,\mathit{\Omega},\mathit{\Omega}_C))$} be the orthogonal projector onto $X_0(\curl\!=\!0,\mathit{\Omega},\mathit{\Omega}_C)$ as defined in \eqref{eq:X0curl0}, where $X_0(\curl,\mathit{\Omega},\mathit{\Omega}_C)$ is provided with the inner product \eqref{eq:Xinner}.
	Then the definition of the solution, see Section~\ref{sec:solution}, yields that for all
	$\bm{F}\!\!\in\! X_0(\curl,\mathit{\Omega},\mathit{\Omega}_C)$,
	\[\begin{aligned}
	\tfrac{\rm d }{{\rm d} t}\int_\mathit{\Omega} & \sigma \bm{A}(t)\cdot (P\bm{F})\, {\rm d}\xi+
	\int_\mathit{\Omega} \nu(\cdot,\|\nabla\times \bm{A}(t)\|_2)(\nabla\times \bm{A}(t))\cdot(\nabla\times (P\bm{F}))\, {\rm d}\xi \\
	= &\;
	\int_\mathit{\Omega} \chi\, \bm{i}(t)\cdot (P\bm{F})\,  {\rm d}\xi\\
	= &\;
	\bm{v}(t)\cdot R^{-1}\int_\mathit{\Omega} \chi^\top (P\bm{F})\,  {\rm d}\xi -
	\tfrac{\rm d }{{\rm d} t}\!\int_\mathit{\Omega} \chi^\top {\bm{A}(t)}\,  {\rm d}\xi\cdot R^{-1}\int_\mathit{\Omega} \chi^\top (P\bm{F})\,  {\rm d}\xi.
	\end{aligned}\]
	Since $\nabla\times (P\bm{F})=0$, this equation reduces to
	\begin{equation}
	\arraycolsep=2pt
	\begin{aligned}
	\tfrac{\rm d }{{\rm d} t} & \left(\int_\mathit{\Omega} \sigma \bm{A}(t)\cdot (P\bm{F})\, {\rm d}\xi+\int_\mathit{\Omega} \chi^\top {\bm{A}(t)}\,  {\rm d}\xi\cdot R^{-1}\int_\mathit{\Omega} \chi^\top (P\bm{F})\,  {\rm d}\xi\right)\\
	&\quad = \bm{v}(t)\cdot R^{-1}\int_\mathit{\Omega} \chi^\top (P\bm{F})\,  {\rm d}\xi.
	\end{aligned}
	\label{eq:weak_new1}
	\end{equation}
	By using that $P\in\mathcal{L}(X_0(\curl,\mathit{\Omega},\mathit{\Omega}_C))$ is an~orthogonal projector onto the space $X_0(\curl\!=\!0,\mathit{\Omega},\mathit{\Omega}_C)$
	with respect to the inner product \eqref{eq:Xinner}
	and, by \cite[Theorem~7.1]{ChReSt21}, $\bm{A}(t)\in X_0(\curl,\mathit{\Omega},\mathit{\Omega}_C)$ for almost all $t\in[0,T)$, we obtain
	\[\begin{aligned}
	\int_\mathit{\Omega} \sigma &\bm{A}(t)\cdot (P\bm{F})\, {\rm d}\xi+\int_\mathit{\Omega} \chi^\top\bm{A}(t)\,  {\rm d}\xi\cdot R^{-1}\int_\mathit{\Omega} \chi^\top (P\bm{F})\,  {\rm d}\xi\\
	&\enskip=\int_\mathit{\Omega} \sigma \bm{A}(t)\cdot (P\bm{F})\, {\rm d}\xi+\int_\mathit{\Omega} \chi^\top \bm{A}(t)\,  {\rm d}\xi\cdot R^{-1}\int_\mathit{\Omega} \chi^\top (P\bm{F})\,  {\rm d}\xi \\
	&\qquad
	+\int_\mathit{\Omega} (\nabla\times \bm{A})\cdot (\underbrace{\nabla\times (P\bm{F}))}_{=0}\, {\rm d}\xi\\
	&\overset{\eqref{eq:Xinner}}{=} \;\langle\bm{A}(t),P\bm{F}\rangle_{X_0(\curl,\mathit{\Omega},\mathit{\Omega}_C)} =\langle P\bm{A}(t),\bm{F}\rangle_{X_0(\curl,\mathit{\Omega},\mathit{\Omega}_C)}\\
	&\enskip = \int_\mathit{\Omega} \sigma (P\bm{A}(t))\cdot \bm{F}\, {\rm d}\xi+\int_\mathit{\Omega} \chi^\top (P\bm{A}(t))\,  {\rm d}\xi \cdot R^{-1}\int_\mathit{\Omega} \chi^\top {\bm{F}}\,  {\rm d}\xi\\
	&\qquad +\int_\mathit{\Omega} \underbrace{(\nabla\times (P\bm{A}(t)))}_{=0}\cdot (\nabla\times\bm{F})\, {\rm d}\xi\\
	&\enskip=\int_\mathit{\Omega} \sigma (P\bm{A}(t))\cdot \bm{F}\, {\rm d}\xi+\int_\mathit{\Omega} \chi^\top (P\bm{A}(t))\,  {\rm d}\xi \cdot R^{-1}\int_\mathit{\Omega} \chi^\top \bm{F}\,  {\rm d}\xi.
	\end{aligned}\]
	Since by Lemma~\ref{lem:orthproj}\,\ref{lem:orthproja}), $P$ extends to a~projector $\widetilde{P}\in\mathcal{L}(X(\mathit{\Omega},\mathit{\Omega}_C))$, we further have
	\begin{align*}
	\bm{v}(t)\cdot R^{-1}\int_\mathit{\Omega} \chi^\top (P\bm{F})\,  {\rm d}\xi
	=&
	\int_\mathit{\Omega} (\chi\, R^{-1}\bm{v}(t)) \cdot  (P\bm{F})\,  {\rm d}\xi
	=
	\langle \chi\, R^{-1} \bm{v}(t), P\bm{F}\rangle_{L^2(\mathit{\Omega};\mathbb{R}^3)}\\
	=&\,
	\langle \chi\, R^{-1} \bm{v}(t), \widetilde{P}\bm{F}\rangle_{L^2(\mathit{\Omega};\mathbb{R}^3)}
	=
	\langle \widetilde{P}^* \chi\, R^{-1} \bm{v}(t), \bm{F}\rangle_{L^2(\mathit{\Omega};\mathbb{R}^3)},
	\end{align*}
	where $\widetilde{P}^*$ is the adjoint of $\widetilde{P}$.
	By using the density of the space $X_0(\curl,\mathit{\Omega},\mathit{\Omega}_C)$ in $X(\mathit{\Omega},\mathit{\Omega}_C)$ and the latter two equations, the integration of \eqref{eq:weak_new1} implies that
	\[t\mapsto \sigma (P\bm{A}(t))+\chi\, R^{-1}\int_\mathit{\Omega}  \chi^\top (P\bm{A}(t))\,  {\rm d}\xi \]
	is a~continuous mapping from $[0,T)$ to $X(\mathit{\Omega},\mathit{\Omega}_C)$ with
	\begin{multline*}
	\left(\sigma (P\bm{A}(t))+\chi\, R^{-1}\int_\mathit{\Omega}  \chi^\top (P\bm{A}(t))\,  {\rm d}\xi\right)-\left(\sigma (P\bm{A}_0)+\chi\, R^{-1}\int_\mathit{\Omega}  \chi^\top (P\bm{A}_0)\,  {\rm d}\xi\right)\\=\widetilde{P}^*\chi\, R^{-1}\int_0^t  \bm{v}(\tau)\,{\rm d}\tau\quad\text{for all } t\in[0,T).
	\end{multline*}
	An application of $\widetilde{P}^*$ to both sides of this equation yields
	\[ \mathcal{T}(P\bm{A}(t)-P\bm{A}_0)=\widetilde{P}^*\chi\,R^{-1}\int_0^t  \bm{v}(\tau)\,{\rm d}\tau\quad\text{for all } t\in[0,T),\]
	where $\mathcal{T}$ is the operator as in Lemma~\ref{lem:orthproj}\,\ref{lem:orthprojc}). Since $\mathcal{T}$ is invertible, we obtain
	\begin{equation}
	P\bm{A}(t)-P\bm{A}_0=\mathcal{T}^{-1}\left(\widetilde{P}^*\chi\,R^{-1}\int_0^t  \bm{v}(\tau)\,{\rm d}\tau\right)\quad\text{for all } t\in[0,T).\label{eq:PAexpr}
	\end{equation}
	Hence, using Lemma~\ref{lem:orthproj}\,\ref{lem:orthproja}) and~\ref{lem:orthprojc}), we obtain for all $t\in[0,T)$,
	\[
	\begin{aligned}
	&\|P\bm{A}(t)\|_{L^2(\mathit{\Omega};\mathbb{R}^3)}\\
	&\quad\;\leq \|P\bm{A}_0\|_{L^2(\mathit{\Omega};\mathbb{R}^3)}+\|\mathcal{T}^{-1}\|\|\widetilde{P}^*\| \|\chi R^{-1/2}\|_{L^2(\mathit{\Omega};\mathbb{R}^{3\times m})} \left\|\int_0^t  R^{-1/2}\bm{v}(\tau)\,{\rm d}\tau\right\|_2\\
	&\quad\;\leq  \sqrt{\tfrac{\gamma\,L_C}{\sigma_C}}\left(
	\|\bm{A}_0\|_{L^2(\mathit{\Omega};\mathbb{R}^3)}+\tfrac{L_C}{\sigma_C}  \|\chi R^{-1/2}\|_{L^2(\mathit{\Omega};\mathbb{R}^{3\times m})} \left\|\int_0^t  R^{-1/2}\bm{v}(\tau)\,{\rm d}\tau\right\|_2\right).
	\end{aligned}\]
	Further, Lemma~\ref{lem:orthproj}\,\ref{lem:orthprojb}) and Proposition~\ref{prop:stateest} imply  for all $t\in[0,T)$,
	\begin{multline*}
	\|(I-P)\bm{A}(t)\|_{L^2(\mathit{\Omega};\mathbb{R}^3)}\leq
	L_1\,\|\nabla\times \bm{A}\|_{L^2(\mathit{\Omega};\mathbb{R}^3)}\\
	\leq L_1\,\sqrt{\tfrac{L_\nu}{m_\nu}}\,\|\nabla\times \bm{A}_0\|_{L^2(\mathit{\Omega};\mathbb{R}^3)}+
	\tfrac{L_1}{\sqrt{2m_\nu}}\,\|R^{-1/2}\bm{v}\|_{L^2([0,t);\R^m)}
	\end{multline*}
	and, thus,
	\[\begin{aligned}
	&\|\bm{A}(t)\|_{L^2(\mathit{\Omega};\mathbb{R}^3)}
	\leq\|P\bm{A}(t)\|_{L^2(\mathit{\Omega};\mathbb{R}^3)}
	+\|(I-P)\bm{A}(t)\|_{L^2(\mathit{\Omega};\mathbb{R}^3)}\\
	&\qquad\leq\sqrt{\tfrac{\gamma\,L_C}{\sigma_C}} \left(
	\|\bm{A}_0\|_{L^2(\mathit{\Omega};\mathbb{R}^3)}+\tfrac{L_C}{\sigma_C}  \|\chi R^{-1/2}\|_{L^2(\mathit{\Omega};\R^{3\times m})} \left\|\int_0^t  R^{-1/2}\bm{v}(\tau)\,{\rm d}\tau\right\|_2\right)\\
	&\qquad\qquad+L_1\sqrt{\tfrac{L_\nu}{m_\nu}}\,\|\nabla\times \bm{A}_0\|_{L^2(\mathit{\Omega};\mathbb{R}^3)}+\tfrac{L_1}{\sqrt{2m_\nu}}\,\|R^{-1/2}\bm{v}\|_{L^2([0,t);\R^m)}.
	\end{aligned}\]
	
	\ref{thm:stateest2b})
	If {$\chi\in L^2(\divg\!=\!0,\mathit{\Omega};\R^3)^{1\times m}$}, then
	it follows from \cite[Theorem~3.17]{ABDG98} that there exists some $\bm{F}\in H_0(\curl,\mathit{\Omega)}^{1\times m}$ with $\chi=\nabla\times F$.
	Using the integration by parts formula for the curl operator, see \cite[eq.~(2.1)]{ChReSt21},
	we obtain that the columns of $\chi$ are orthogonal with respect to the inner product in $L^2(\mathit{\Omega};\R^{3})$ to all \mbox{$\bm{C}\in {X_0(\curl\!=\!0,\mathit{\Omega},\mathit{\Omega}_C)}$}.
	In other words,
	\[\chi\in ((\im\widetilde{P})^\bot)^{1\times m}=(\ker \widetilde{P}^*)^{1\times m},\]
	which gives $\widetilde{P}^*\chi=0$. Then \eqref{eq:PAexpr} reduces to
	$P\bm{A}(t)=P\bm{A}_0$ for all $t\in[0,T)$.
	Now proceeding as in the previous case, we obtain
	\[\begin{aligned}
	\|\bm{A}(t)\|_{L^2(\mathit{\Omega};\mathbb{R}^3)}
	\leq&\;\|P\bm{A}(t)\|_{L^2(\mathit{\Omega};\mathbb{R}^3)}
	+\|(I-P)\bm{A}(t)\|_{L^2(\mathit{\Omega};\mathbb{R}^3)}\\
	\leq&\;\sqrt{\tfrac{\gamma\,L_C}{\sigma_C}}\,\|\bm{A}_0\|_{L^2(\mathit{\Omega};\mathbb{R}^3)}
	\\&\;+L_1\,\sqrt{\tfrac{L_\nu}{m_\nu}}\,\|\nabla\times \bm{A}_0\|_{L^2(\mathit{\Omega};\mathbb{R}^3)}+
	\tfrac{L_1}{\sqrt{2m_\nu}}\,\|R^{-1/2}\bm{v}\|_{L^2([0,t);\R^m)}.
	\end{aligned}\]
\end{proof}

\begin{remark}
	Note that the divergence-freeness condition in Theorem~\textup{\ref{thm:stateest2}\,\ref{thm:stateest2b})} is guaranteed, if the support of $\chi$ does not meet the interface between the conducting and non-conducting subdomains, i.e.,
	$\mbox{supp}(\chi) \subset \mathit{\Omega}_C\cup \mathit{\Omega}_I$. In this case, the condition $\chi\in X(\mathit{\Omega},\mathit{\Omega}_C)^{1\times m}$ is even equivalent to
	$\chi\in L^2(\divg\!=\!0,\mathit{\Omega};\mathbb{R}^3)^{1\times m}$.
\end{remark}

\subsection{The free dynamics}\label{sec:uncont}

By {\em free dynamics}, we mean the solution behavior of the coupled MQS system \eqref{eq:MQS} in which \eqref{eq:MQS} no voltage is applied, i.e., $\bm{v}=0$. Our goal is now to analyse their asymptotic behavior for $t\to\infty$.

The following theorem shows that the magnetic energy $E(\bm{A}(t))$ as well as the \mbox{$L^2$-norm} of $\nabla\times \bm{A}(t)$ can be bounded from above by an~exponentially decaying  function provided $\bm{v}\equiv 0$.

\begin{theorem}\label{thm:expstab}
	Assume that $\mathit{\Omega}\subset\mathbb{R}^3$ with a~subdomain $\mathit{\Omega}_C$ satisfies Assumption~\textup{\ref{ass:omega}}. Further, let
	Assumptions~\textup{\ref{ass:material}} and~\textup{\ref{ass:init}} be fulfilled, and
	$\bm{A}_{0}\!\in\! X_0(\curl,\mathit{\Omega},\mathit{\Omega}_C)$. Moreover,
	let $E$ be the magnetic energy
	as defined in \eqref{eq:varphiA}, and let $(\bm{A},\bm{i})$
	be a~solution of the coupled MQS system \eqref{eq:MQS} with $\bm{v}\equiv 0$. Then for $\sigma_C$, $m_\nu$ and $L_\nu$  as in Assumption~\textup{\ref{ass:material}},
	$\gamma$ as in \eqref{eq:cconst}, and
	\begin{equation}\omega:=\tfrac{m_\nu^2}{\gamma\, L_\nu},\label{eq:omega}\end{equation}
	it holds for all $t\geq 0$ that
	\begin{align}
	E(\bm{A}(t))\leq&\, e^{-2\,\omega t}E(\bm{A}_0),\label{eq:MQSexpstab1}\\
	\|\nabla\times \bm{A}(t)\|_{L^2(\mathit{\Omega};\mathbb{R}^3)}\leq&\, \sqrt{\tfrac{L_\nu}{m_\nu}}\,
	e^{-\omega t}\,\|\nabla\times {\bm{A}_0}\|_{L^2(\mathit{\Omega};\mathbb{R}^3)}.\label{eq:MQSexpstab2}
	\end{align}
\end{theorem}

\begin{proof}
	Consider the operator
	\begin{align*}
	\mathcal{F}: \;\;X(\mathit{\Omega},\mathit{\Omega}_C)\times \R^m&\,\to\,X(\mathit{\Omega},\mathit{\Omega}_C),\\
	(\bm{A},\bm{i})&\,\mapsto\,\sqrt{\sigma} \bm{A}+\chi\, R^{-1/2}\,\bm{i}.
	\end{align*}
	Then $\mathcal{F}$ is {linear and} bounded with
	$\|\mathcal{F}\|\leq {\sqrt{\gamma}}$.
	Further, $\im\mathcal{F}$ is closed, as it is the sum of the closed space $L^2(\mathit{\Omega_C};\mathbb{R}^3)$ and a~finite-dimensional subspace of $X(\mathit{\Omega},\mathit{\Omega}_C)$.
	Therefore, we conclude from \cite[Theorem~9.3.3]{BenIsGre03}
	that $\mathcal{F}$ has a~bounded Moore-Penrose inverse \mbox{$\mathcal{F}^+:X(\mathit{\Omega},\mathit{\Omega}_C)\to X(\mathit{\Omega},\mathit{\Omega}_C)\times \R^m$} such that $\mathcal{F}\mathcal{F}^+$ and $\mathcal{F}^+\mathcal{F}$ are the ortho\-go\-nal projectors onto $\im \mathcal{F}$ and $\im \mathcal{F}^*$, respectively. Using $\|\mathcal{F}\|\leq {\sqrt{\gamma}}$,
	we obtain for all $\bm{F}\in\im \mathcal{F}$,
	\begin{equation}
	\|\bm{F}\|_{{L^2(\mathit{{\Omega}};\mathbb{R}^3)}}=
	\|\mathcal{F}\mathcal{F}^+\bm{F}\|_{{L^2(\mathit{{\Omega}};\mathbb{R}^3)}}
	\leq {\sqrt{\gamma}}\,\|\mathcal{F}^+\bm{F}\|_{L^2(\mathit{{\Omega}};\mathbb{R}^3){\times\R^m}}.
	\label{eq:E-diss}
	\end{equation}
	Let $(\bm{A},\bm{i})$ be a~solution of \eqref{eq:MQS} with $\bm{v}\equiv 0$. Then it follows from \eqref{eq:MQSsol} that for almost all $t\geq 0$,
	\begin{equation}\label{eq:pSAimE}
	\begin{aligned}
	\nabla \times \left(\nu(\cdot,\|\nabla \times \bm{A}(t)\|_2)
	\nabla \times \bm{A}(t)\right)& =-\tfrac{\rm d}{{\rm d} t}\left(\sigma\bm{A}(t)\right) +  \chi\, \bm{i}(t)\\
	& =-\mathcal{F}\tfrac{\rm d}{{\rm d}t} \mathcal{F}^*\bm{A}(t)\in\im\mathcal{F},
	\end{aligned}
	\end{equation}
	and Assumption~\ref{ass:material}\,\ref{ass:material2}) gives
	\[\langle\nabla \times \bm{A}(t), \nu(\cdot,\|\nabla \times \bm{A}(t)\|_2)\nabla \times \bm{A}(t)\rangle_{L^2(\mathit{{\Omega}};\mathbb{R}^3)}\geq m_\nu\|\nabla\times {\bm{A}}(t)\|^2_{L^2(\mathit{\Omega};\mathbb{R}^3)}.\]
	Further, using the integration by parts formula
	for the weak curl operator together with the fact that $\bm{A}(t)\in X_0(\curl,\mathit{\Omega},\mathit{\Omega}_C)$ for almost all $t\geq0$, we obtain
	\begin{equation}\label{eq:pSAdiss}
	\|\nabla \times \left(\nu(\cdot,\|\nabla \times \bm{A}(t)\|_2)\nabla \times \bm{A}(t)\right)\|_{L^2(\mathit{{\Omega}};\mathbb{R}^3)}\geq m_\nu\|\nabla\times {\bm{A}}(t)\|_{L^2(\mathit{\Omega};\mathbb{R}^3)}.
	\end{equation}
	Then it follows from \eqref{eq:E-diss} and \eqref{eq:pSAdiss} that for almost all $t\geq 0$,
	\begin{equation}\label{eq:E-est}
	\begin{aligned}
	\left\|\mathcal{F}^+(\nabla \right. &\left.\times (\nu(\cdot,\|\nabla \times \bm{A}(t)\|_2)
	\nabla \times \bm{A}(t)))\right\|_{L^2(\mathit{{\Omega}};\mathbb{R}^3)\times\R^m}\\
	& \geq \tfrac{1}{{\sqrt{\gamma}}}\, \left\|\nabla \times \left(\nu(\cdot,\|\nabla \times \bm{A}(t)\|_2)
	\nabla \times \bm{A}(t)\right)\right\|_{L^2(\mathit{\Omega}_C;\mathbb{R}^3)}\\
	& \geq \tfrac{m_\nu}{{\sqrt{\gamma}}}\, \left\|\nabla \times \bm{A}(t)\right\|_{L^2(\mathit{\Omega}_C;\mathbb{R}^3)}.
	\end{aligned}\end{equation}
	Let {$\omega$ be as in \eqref{eq:omega}}. Using the energy balance \eqref{eq:MQSpass}
	and the relation $\mathcal{F}^*=\mathcal{F}^+\mathcal{F}\mathcal{F}^*$, we obtain, by invoking $\bm{v}\equiv 0$, that for all $0\leq t_0<t<\infty$,
	\begin{align*}
	E(\bm{A}&(t_1))-E(\bm{A}(t_0))\\
	&\;=\;-\int_{t_0}^{t_1}\|\tfrac{\rm d}{{\rm d}\tau} \sqrt{\sigma}\bm{A}(\tau))\|^2_{L^2(\mathit{\Omega};\mathbb{R}^3)}{\rm d}\tau-\int_{t_0}^{t_1}\langle\bm{i}(\tau), R\,\bm{i}(\tau)\rangle_{2}\,{\rm d}\tau\\
	&\!\overset{\eqref{eq:MQSsol}}{=} -\int_{t_0}^{t_1}\|\tfrac{\rm d}{{\rm d}\tau}
	\sqrt{\sigma}\bm{A}(\tau))\|^2_{L^2(\mathit{\Omega};\mathbb{R}^3)}{\rm d}\tau
	-{\int_{t_0}^{t_1}} \left\|\tfrac{\rm d }{{\rm d} \tau} R^{-1/2}\int_\mathit{\Omega} \chi^\top\bm{A}(\tau)\, {\rm d}\xi\right\|_2^2\,{\rm d}\tau\\
	&\;=\;-\int_{t_0}^{t_1}\|\tfrac{\rm d}{{\rm d}\tau} \mathcal{F}^*\bm{A}(\tau)\|^2_{L^2(\mathit{\Omega};\mathbb{R}^3)\times\R^m}{\rm d}\tau\\
	&\;=\;-\int_{t_0}^{t_1}\|\mathcal{F}^+\mathcal{F}\tfrac{\rm d}{{\rm d}\tau} \mathcal{F}^*\bm{A}(\tau))\|^2_{L^2(\mathit{\Omega};\mathbb{R}^3)\times\R^m}{\rm d}\tau\\
	&\!\overset{\eqref{eq:pSAimE}}{=}
	-\int_{t_0}^{t_1}\left\|\mathcal{F}^+
	\left(\nabla \times \left(\nu(\|\nabla \times \bm{A}(\tau)\|_2)
	\nabla \times \bm{A}(\tau)\right)\right)\right\|^2_{L^2(\mathit{\Omega};\mathbb{R}^3)\times\R^m}{\rm d}\tau\\
	&\!\overset{\eqref{eq:E-est}}{\leq}
	-\frac{m_\nu^2}{\gamma}\, \int_{t_0}^{t_1}\left\|\nabla \times \bm{A}(\tau)\right\|^2_{L^2(\mathit{\Omega};\mathbb{R}^3)}{\rm d}\tau\\
	&\!\overset{\eqref{eq:equivalence}}{\leq}
	- 2\,\omega\,\int_{t_0}^{t_1} E(\bm{A}(\tau)){\rm d}\tau
	\end{align*}
	with $\omega$ as in \eqref{eq:omega}.
	By a~division of the above inequality by $t_1-t_0$ and then taking the limit $t_1\to t_0$, we obtain that the weak derivative of $t\mapsto  E(\bm{A}(t))$ fulfills the differential inequality
	\begin{equation}\tfrac{\rm d}{{\rm d}t} E(\bm{A}(t))\leq -2\,\omega \, E(\bm{A}(t)).\label{eq:Gronwall}\end{equation}
	Then Gr\"onwall's inequality for the weak derivative \cite[Lemma~IV.4.1]{Showalter96} gives rise to \eqref{eq:MQSexpstab1}. The estimate \eqref{eq:MQSexpstab2} can be concluded from \eqref{eq:MQSexpstab1} by further using the inequalities in \eqref{eq:equivalence}, and subsequently taking the square root.
\end{proof}

\newpage
\begin{remark}\
	\begin{enumerate}[\rm a)]
		\item The inequality \eqref{eq:Gronwall} shows that the scalar function $t\mapsto  E(\bm{A}(t))$ is strictly decaying unless $\nabla\times\bm{A}(t){= 0}$. This is not surprising from a~physical point of view, as, by $\bm{v}\equiv 0$, no external energy is put into the system.
		\item Let us briefly consider the case where the free MQS system is initialized with \linebreak
		\mbox{$\bm{A}_0\in X(\mathit{\Omega},\mathit{\Omega}_C)$}, which is not necessarily in $H_0(\curl,\mathit{\Omega})$. By \textup{\cite[Theorem~7.1]{ChReSt21}} on the existence and regularity properties of the solutions of the coupled MQS system \eqref{eq:MQS},
		we have $\bm{A}(t)\in X_0(\curl,\mathit{\Omega},\mathit{\Omega}_C)$ for almost all $t>0$. Further, for each finite interval $[0,T]$, the functions
		$t\mapsto  E(\bm{A}(t))$ and $t\mapsto \|\nabla\times \bm{A}(t)\|_{L^2(\mathit{\Omega};\mathbb{R}^3)}$
		can be bounded by a~constant times $\tfrac{1}{t}$ and $\tfrac{1}{\sqrt{t}}$, respectively. As a~consequence, there exists some constants $M_1,M_2>0$ such that  the solution
		$(\bm{A},\bm{i})$ of \eqref{eq:MQS} with $\bm{v}\equiv 0$ satisfies for all $t>0$,
		\[
		E(\bm{A}(t))\leq M_1 \,(1+\tfrac1t)\, e^{-2\omega t},\quad
		\|\nabla\times \bm{A}(t)\|_{L^2(\mathit{\Omega};\mathbb{R}^3)}\leq  M_2\,(1+\tfrac1{\sqrt{t}})\, e^{-\omega t}.
		\]
	\end{enumerate}
\end{remark}

Finally, we derive the estimates for the $L^2$-norm of the magnetic vector potential $\bm{A}(t)$ of the uncontrolled coupled MQS system \eqref{eq:MQS}.

\begin{theorem}\label{th:estL2uncontr}
	Assume that $\mathit{\Omega}\subset\mathbb{R}^3$ with a~subdomain $\mathit{\Omega}_C$
	satisfies Assumption~\textup{\ref{ass:omega}}. Further, let
	Assumptions~\textup{\ref{ass:material}} and~\textup{\ref{ass:init}} be fulfilled, and
	$\bm{A}_{0}\in X_0(\curl,\mathit{\Omega},\mathit{\Omega}_C)$ and $\bm{v}\equiv 0$. Moreover, let $E$ be the magnetic energy as defined in \eqref{eq:varphiA}, and let $(\bm{A},\bm{i})$
	be a~solution of the coupled MQS system \eqref{eq:MQS}.
	Then for $\sigma_C$, $L_\nu$, $m_\nu$ as in Assumption~\textup{\ref{ass:material}},
	$\gamma$ as in \eqref{eq:cconst}, and $\omega$ as in \eqref{eq:omega}, it holds  for all $t\geq 0$,
	\begin{equation}
	\|\bm{A}(t)\|_{L^2(\mathit{\Omega};\mathbb{R}^3)}
	\leq \sqrt{\tfrac{\gamma\,L_C}{\sigma_C}}\,\|\bm{A}_0\|_{L^2(\mathit{\Omega};\mathbb{R}^3)}+
	L_1 \,\sqrt{\tfrac{L_\nu}{m_\nu}}\,  e^{-\omega t}\|\nabla\times \bm{A}_0\|_{L^2(\mathit{\Omega};\mathbb{R}^3)}.\label{eq:stateest2b1}
	\end{equation}
	If, additionally, the initial value fulfills for all $\bm{F}\in X_0(\curl\!=\!0,\mathit{\Omega},\mathit{\Omega}_C)$,
	\begin{equation}\label{eq:orthcond}
	\int_{\mathit{\Omega}} \sigma\bm{A}_0\cdot\bm{F}\,  {\rm d}\xi+\int_\mathit{\Omega}
	\chi^\top \bm{A}_0\,  {\rm d}\xi\cdot R^{-1} \int_\mathit{\Omega}
	\chi^\top \bm{F}\,  {\rm d}\xi=0,
	\end{equation}
	then for all $t\geq 0$,
	\begin{equation}
	\|\bm{A}(t)\|_{L^2(\mathit{\Omega};\mathbb{R}^3)}
	\leq
	L_1 \,\sqrt{\tfrac{L_\nu}{m_\nu}}\,  e^{-\omega t}\|\nabla\times \bm{A}_0\|_{L^2(\mathit{\Omega};\mathbb{R}^3)}.
	\label{eq:stateest2b2}
	\end{equation}
\end{theorem}

\begin{proof}
	Let $\widetilde{P}\in\mathcal{L}(X(\mathit{\Omega},\mathit{\Omega}_C))$ be the projector as in Lemma~\ref{lem:orthproj}\,\ref{lem:orthproja}). Then by using the argumentation as in the proof of Theorem~\ref{thm:stateest2}\,\ref{thm:stateest2a})
	and invoking \mbox{$\bm{v}\equiv 0$}, we obtain that $\widetilde{P}\bm{A}(t)=\widetilde{P}\bm{A}_0$ for all $t\geq0$. Thus, by further using that
	\mbox{$\bm{A}(t)\in X_0(\curl,\mathit{\Omega},\mathit{\Omega}_C)$}, we obtain for almost all $t\geq0$,
	\begin{equation}
	\begin{aligned}
	\|\bm{A}(t)\|_{L^2(\mathit{\Omega};\mathbb{R}^3)}
	&\;\leq\;\|\widetilde{P}\bm{A}(t)\|_{L^2(\mathit{\Omega};\mathbb{R}^3)}+\|(I-\widetilde{P})\bm{A}(t)\|_{L^2(\mathit{\Omega};\mathbb{R}^3)}\\
	&\;=\;\|\widetilde{P}\bm{A}_0\|_{L^2(\mathit{\Omega};\mathbb{R}^3)}+\|(I-{P})\bm{A}(t)\|_{L^2(\mathit{\Omega};\mathbb{R}^3)}\\
	&\!\overset{\eqref{eq:L1const}}{\leq}\|\widetilde{P}\bm{A}_0\|_{L^2(\mathit{\Omega};\mathbb{R}^3)}+L_1\|\nabla\times\bm{A}(t)\|_{L^2(\mathit{\Omega};\mathbb{R}^3)}\\
	&\!\overset{\eqref{eq:MQSexpstab2}}{\leq}\|\widetilde{P}\bm{A}_0\|_{L^2(\mathit{\Omega};\mathbb{R}^3)}+L_1 \,\sqrt{\tfrac{L_\nu}{m_\nu}}\,  e^{-\omega t}\|\nabla\times \bm{A}_0\|_{L^2(\mathit{\Omega};\mathbb{R}^3)}.
	\end{aligned}\label{eq:stateL2est1}\end{equation}
	Then \eqref{eq:stateest2b1} follows by using the bound $\|\widetilde{P}\|\leq \sqrt{\tfrac{\gamma\,L_C}{\sigma_C}}$ from Lemma~\ref{lem:orthproj}\,\ref{lem:orthproja}).
	
	On the other hand, if $\bm{A}_0$ satisfies \eqref{eq:orthcond}, then $\bm{A}_0$ is orthogonal to all elements of $X_0(\curl\!=\!0,\mathit{\Omega},\mathit{\Omega}_C)$ with respect to the inner product \eqref{eq:Xinner}. Since $P$ is an orthogonal projector with respect to that inner product, we have $\widetilde{P}\bm{A}_0={P}\bm{A}_0=0$, and, hence, \eqref{eq:stateL2est1} immediatelly implies \eqref{eq:stateest2b2}.
\end{proof}


\section{Conclusion}

We have considered a quasilinear magneto-quasistatic approximation of Maxwell's
equations, which is furthermore coupled with an integral equation. By employing
the magnetic energy, we have shown that this system is passive and admits a~representation as
a~port-Hamiltonian system.
Further, we have derived estimates of the state and output of the system by means of the initial value and the input. A~special emphasis in the solution estimates is placed on the free system with the zero input voltage,
where we have shown that the magnetic energy decays exponentially.


\providecommand{\href}[2]{#2}
\providecommand{\arxiv}[1]{\href{http://arxiv.org/abs/#1}{arXiv:#1}}
\providecommand{\url}[1]{\texttt{#1}}
\providecommand{\urlprefix}{URL }

%

\end{document}